\documentclass[a4paper,reqno]{amsart}
\usepackage[all]{xy} 
  \newdir{ >}{{}*!/-5pt/@{>}}
\usepackage{diagrams,amssymb}
\diagramstyle[height=2em,width=2em,midshaft,labelstyle=\scriptstyle]
\newarrow{To}----{->}
\newarrow{Epi}----{->>}
\newarrow{Mono}>---{->}
\newarrow{Iso}>---{->>}
\newarrow{Mapsto}|---{->}
\newarrow{Igual}=====
\newarrow{Dashto}{}{dash}{}{dash}{->}

\newcommand{\mynote}[1]{\noindent{\textbf{[#1]}}}

\newcommand{\chr}{\textup{char}}

\newcommand{\longdisplay}[2][]{\par\noindent\parbox{8truemm}{\hfill}%
        \hfill\parbox{105truemm}{#2}\hfill%
        \def\test{#1}\ifx\test\empty{\parbox{8truemm}{\hfill}}%
        \else{\parbox{8truemm}{\hfill(#1)}}\fi\par\noindent}

\newcommand{\Z}{{\mathbb Z}}
\newcommand{\Q}{{\mathbb Q}}

\newcommand{\C}{{\mathbb C}}

\newcommand{\F}{{\mathbb F}}

\newcommand{\pcom}{{}_{p}^{\wedge}}

\DeclareMathAlphabet\EuR{U}{eur}{m}{n}
\SetMathAlphabet\EuR{bold}{U}{eur}{b}{n}

\newcommand{\defeq}{\overset{\text{\textup{def}}}{=}}
\newcommand{\gen}[1]{{\langle}#1{\rangle}}
\renewcommand{\:}{\colon}

\newcommand{\calc}{\mathcal{C}}

\newcommand{\calf}{\mathcal{F}}

\newcommand{\caly}{\mathcal{Y}}
\newcommand{\calz}{\mathcal{Z}}

\newcommand{\orb}{\mathcal{O}}

\newcommand{\Id}{\operatorname{Id}\nolimits}
\newcommand{\incl}{\operatorname{incl}\nolimits}
\newcommand{\proj}{\operatorname{pr}\nolimits}

\let\oldcirc=\circ
\renewcommand{\circ}{\mathchoice
    {\mathbin{\scriptstyle\oldcirc}}{\mathbin{\scriptstyle\oldcirc}}
    {\mathbin{\scriptscriptstyle\oldcirc}}
    {\mathbin{\scriptscriptstyle\oldcirc}}}

\newcommand{\hclim}[1]{\setbox1=\hbox{\rm hocolim}
    \setbox2=\hbox to \wd1{\rightarrowfill} \ht2=0pt \dp2=-1pt
    \mathop{\vtop{\baselineskip=5pt\box1\box2}}
    _{#1}}

\newcommand{\map}{\operatorname{map}\nolimits}
\newcommand{\Hom}{\operatorname{Hom}\nolimits}
\newcommand{\Rep}{\operatorname{Rep}\nolimits}
\newcommand{\Iso}{\operatorname{Iso}\nolimits}

\newcommand{\End}{\operatorname{End}\nolimits}

\newcommand{\Mor}{\operatorname{Mor}\nolimits}

\renewcommand{\Im}{\operatorname{Im}\nolimits}

\newcommand{\Out}[2]{\operatorname{Out}_{#1}(#2)}
\newcommand{\Aut}{\operatorname{Aut}}

\newcommand{\lcm}[2]{\operatorname{lcm}(#1,#2)}
\renewcommand{\gcd}[2]{\operatorname{gcd}(#1,#2)}

\newcommand{\longleft}[1]{\;{\leftarrow%
\count255=0 \loop \mathrel{\mkern-6mu}%
    \relbar\advance\count255 by1\ifnum\count255<#1\repeat}\;}
\newcommand{\longright}[1]{\;{\count255=0 \loop \relbar\mathrel{\mkern-6mu}%
    \advance\count255 by1\ifnum\count255<#1\repeat\rightarrow}\;}
\newcommand{\Right}[2]{\overset{#2}{\longright#1}}
\newcommand{\RIGHT}[3]{\mathrel{\mathop{\kern0pt\longright#1}
        \limits^{#2}_{#3}}}
\newcommand{\Left}[2]{{\buildrel #2 \over {\longleft#1}}}
\newcommand{\LEFT}[3]{\mathrel{\mathop{\kern0pt\longleft#1}\limits^{#2}_{#3}}
}
\newcommand{\dRIGHT}[3]{\mathrel{%
   \mathop{\vcenter{\baselineskip=0pt\hbox{$\kern0pt\longright#1$}%
   \hbox{$\kern0pt\longright#1$}}}\limits^{#2}_{#3}}}
\newcommand{\LRIGHT}[3]{\mathrel{%
   \mathop{\vcenter{\baselineskip=0pt\hbox{$\kern0pt\longleft#1$}%
   \hbox{$\kern0pt\longright#1$}}}\limits^{#2}_{#3}}}
\newcommand{\RLEFT}[3]{\mathrel{%
   \mathop{\vcenter{\baselineskip=0pt\hbox{$\kern0pt\longright#1$}%
   \hbox{$\kern0pt\longleft#1$}}}\limits^{#2}_{#3}}}
\newcommand{\onto}[1]{\;{\count255=0 \loop \relbar\mathrel{\mkern-6mu}%
    \advance\count255 by1
    \ifnum\count255<#1 \repeat \twoheadrightarrow}\;}
\newcommand{\Onto}[2]{\overset{#2}{\onto#1}}

\let\txtg=\xxx
\let\txtr=\xxx
\let\txtb=\xxx

\newtheorem{thm}{\txtb{Theorem}}[section]
\newtheorem{prop}[thm]{\txtb{Proposition}}
\newtheorem{cor}[thm]{\txtb{Corollary}}
\newtheorem{lemma}[thm]{\txtb{Lemma}}

\newtheorem{Th}{Theorem}

\newtheorem{defn}[thm]{\txtr{Definition}}
\newtheorem{rmk}[thm]{\txtg{Remark}}
\newtheorem{exmp}[thm]{\txtg{Example}}

\newenvironment{smallpmatrix}
                 {\bigl(\begin{smallmatrix}}{\end{smallmatrix}\bigr)}

\newcommand{\fqbar}{\widebar{\F}_q}

\newcommand{\longline}{\smallskip\centerline{\hbox to 
5cm{\hrulefill}}\smallskip}

\renewenvironment{enumerate}{\begin{list}%
{\labelenumi}
{\usecounter{enumi}%
\setlength{\itemindent}{0pt}%
\settowidth{\labelwidth}{\labelenumi}%
\addtolength{\labelwidth}{\labelsep}%
\setlength{\leftmargin}{\labelsep}%
\addtolength{\leftmargin}{\labelwidth}%
\setlength{\listparindent}{0pt}%
\setlength{\itemsep}{6pt}%
\setlength{\parsep}{0pt}%
\setlength{\topsep}{6pt}%
}}{\end{list}}

\renewenvironment{itemize}
{\begin{itemiz}\setlength{\itemsep}{6pt}\setlength{\itemindent}{-20pt}}
{\end{itemiz}}

\def\beq#1\eeq{\begin{equation*}#1\end{equation*}}
\def\beqq#1\eeqq{\begin{equation}#1\end{equation}}

\newcommand{\pageskip}[1]{{\count255=1 \loop \phantom{.}\vfill\eject%
\advance\count255 by1\ifnum\count255<#1\repeat}}

\newcommand{\lie}[3]{\def\test{#2}\def\tst{G}\ifx\test\tst{{}^{#1}#2_{#3}}
\else{{}^{#1}\!#2_{#3}}\fi}

\newcommand{\bb}{\mathfrak{b}}  

\renewcommand{\gg}{\mathbf{G}}
\renewcommand{\*}{\,\lower6pt\hbox{\Large{\textup{*}}}\,}

\newcommand{\Syl}{\textup{Syl}}

\def\syl#1#2{\Syl_{#1}(#2)}
\def\sylp#1{\syl{p}{#1}}

\def\z(#1,#2){\calz^{#2}_{#1}}
\def\zz(#1,#2,#3/#4){\calz^{#3:#4}_{#1:#2}}
\def\zbar(#1,#2,#3/#4){\def\test{#2}
        \ifx\test\empty{\widebar{\calz}^{#3:#4}_{#1}}%
        \else{\widebar{\calz}^{#3:#4}_{#1:#2}}\fi}
\def\y(#1,#2){\caly_{#2}^{#1}}
\def\yy(#1,#2,#3){(\caly_{#2}^{#1})^{\vphantom{|}}_{#3}}

\newcommand{\sub}{\mathcal{S}}

\newcommand{\subp}[2][]{\def\test{#1}\ifx\test\empty{\sub_p(#2)}%
        \else{\sub_p^{[#1]}(#2)}\fi}
\newcommand{\orbp}[2][]{\def\test{#1}\ifx\test\empty{\orb_p(#2)}%
        \else{\orb_p^{[#1]}(#2)}\fi}

\newcommand{\widebar}[1]
      {\overset{{\mskip3mu\leaders\hrule height0.4pt\hfill\mskip3mu}}{#1}
      \vphantom{#1}}
\newcommand{\higherlim}[2]{\displaystyle\setbox1=\hbox{\rm lim}
        \setbox2=\hbox to \wd1{\leftarrowfill} \ht2=0pt \dp2=-1pt
        \setbox3=\hbox{$\scriptstyle{#1}$}
        \def\test{#1}\ifx\test\empty
        \mathop{\mathop{\vtop{\baselineskip=5pt\box1\box2}}}\nolimits^{#2}
        \else
        \ifdim\wd1<\wd3
        \mathop{\hphantom{^{#2}}\vtop{\baselineskip=5pt\box1\box2}^{#2}}_{#1}
        \else
        \mathop{\mathop{\vtop{\baselineskip=5pt\box1\box2}}_{#1}}%
        \nolimits^{#2}
        \fi\fi}

\title{Equivalences between fusion systems of finite groups of Lie type}

\author{Carles Broto}
\address{Departament de Matem\`atiques, Universitat Aut\`onoma de 
Barcelona, E--08193 Bellaterra, Spain}
\email{broto@mat.uab.es}
\thanks{C. Broto is partially supported by MEC grant MTM2007--61545}
\author{Jesper M. M\o{}ller}
\address{Matematisk Institut, Universitetsparken 5, DK--2100 K\o{}benhavn, 
Denmark}
\email{moller@math.ku.dk}
\thanks{}
\author{Bob Oliver}
\address{LAGA, Institut Galil\'ee, Av. J-B Cl\'ement, F--93430 
Villetaneuse, France}
\email{bobol@math.univ-paris13.fr}
\thanks{B. Oliver is partially supported by UMR 7539 of the CNRS}

\renewcommand{\gg}{\mathbb{G}}
\newcommand{\Spin}{\textup{Spin}}


\begin{document}

\begin{abstract}
We prove, for certain pairs $G,G'$ of finite groups of Lie type, that the 
$p$-fusion systems $\calf_p(G)$ and $\calf_p(G')$ are equivalent.  In 
other words, there is an isomorphism between a Sylow $p$-subgroup of $G$ 
and one of $G'$ which preserves $p$-fusion.  This occurs, for example, 
when $G=\gg(q)$ and $G'=\gg(q')$ for a simple Lie ``type'' $\gg$, and $q$ 
and $q'$ are prime powers, both prime to $p$, which generate the same 
closed subgroup of $p$-adic units.  Our proof uses homotopy theoretic 
properties of the $p$-completed classifying spaces of $G$ and $G'$, and we 
know of no purely algebraic proof of this result.
\end{abstract}

\maketitle

When $G$ is a finite group and $p$ is a prime, the fusion system 
$\calf_p(G)$ is the category whose objects are the $p$-subgroups of $G$, 
and whose morphisms are the homomorphisms between subgroups induced by 
conjugation in $G$.  If $G'$ is another finite group, then $\calf_p(G)$ and 
$\calf_p(G')$ are isotypically equivalent if there is an equivalence of 
categories between them which commutes, up to natural isomorphism of 
functors, with the forgetful functors from $\calf_p(-)$ to the category of 
groups.  Alternatively, $\calf_p(G)$ and $\calf_p(G')$ are isotypically 
equivalent if there is an isomorphism between Sylow $p$-subgroups of $G$ 
and of $G'$ which is ``fusion preserving'' in the sense of Definition 
\ref{defn:fuspres} below.  

The goal of this paper is to use methods from homotopy theory to prove that 
certain pairs of fusion systems of finite groups of Lie type are 
isotypically equivalent.  Our main result is the following theorem. 

\begin{Th}\label{ThA}
Fix a prime $p$, a connected reductive integral group scheme $\gg$, and a 
pair of prime powers $q$ and $q'$ both prime to $p$.  Then the following 
hold, where ``$\simeq$'' always means isotypically equivalent.
\begin{enumerate}\renewcommand{\labelenumi}{\textup{(\alph{enumi})}}
\item $\calf_p(\gg(q))\simeq\calf_p(\gg(q'))$ if 
$\widebar{\gen{q}}=\widebar{\gen{q'}}$ as subgroups of $\Z_p^\times$. 
\item If $\gg$ is of type $A_n$, $D_n$, or $E_6$, and $\tau$ is a graph 
automorphism of $\gg$, then 
$\calf_p(^\tau\!\gg(q))\simeq\calf_p(^\tau\!\gg(q'))$ if 
$\widebar{\gen{q}}=\widebar{\gen{q'}}$ as subgroups of $\Z_p^\times$. 
\item If the Weyl group of $\gg$ contains an element which acts on the 
maximal torus by inverting all elements, then 
$\calf_p(\gg(q))\simeq\calf_p(\gg(q'))$ (or 
$\calf_p(^\tau\!\gg(q))\simeq\calf_p(^\tau\!\gg(q'))$ for $\tau$ as in (b)) if 
$\widebar{\gen{-1,q}}=\widebar{\gen{-1,q'}}$ as subgroups of $\Z_p^\times$. 
\item If $\gg$ is of type $A_n$, $D_n$ for $n$ odd, or $E_6$, and $\tau$ is 
a graph automorphism of $\gg$ of order two, then 
$\calf_p(^\tau\!\gg(q))\simeq\calf_p(\gg(q'))$ if 
$\widebar{\gen{-q}}=\widebar{\gen{q'}}$ as subgroups of $\Z_p^\times$. 
\end{enumerate}
\end{Th}

\noindent Here, in all cases, $\gg(q)$ means the fixed subgroup of the 
field automorphism $\psi^q$ acting on $\gg(\fqbar)$, and $^\tau\!\gg(q)$ 
means the fixed subgroup of $\tau\psi^q$ acting on $\gg(\fqbar)$.  

We remark here that this theorem does not apply when comparing fusion 
systems of $SO_n^\pm(q)$ and $SO_n^\pm(q')$ for even $n$, at least not 
when $q$ or $q'$ is a power of 2, since $SO_n(K)$ is not connected when 
$K$ is algebraically closed of characteristic two.  Instead, one must 
compare the groups $\Omega_n^\pm(-)$.  For example, for even $n\ge4$, 
$\Omega_n^+(4)$ and $\Omega_n^+(7)$ have equivalent $3$-fusion systems, 
while $SO_n^+(4)$ and $SO_n^+(7)$ do not.

Points (a)--(c) of Theorem \ref{ThA} will be proven in Proposition 
\ref{pr:G(q)=G(q')}, where we deal with the more general situation where 
$\gg$ is reductive (thus including cases such as $\gg=GL_n$).  Point (d) 
will be proven as Proposition \ref{pr:G+(q)=G-(q')}.  In all cases, this 
will be done by showing that the $p$-completed classifying spaces of the 
two groups are homotopy equivalent.  A theorem of Martino 
and Priddy (Theorem \ref{th:easyMP} below) then implies that the fusion 
systems are isotypically equivalent.

Since $p$-completion of spaces plays a central role in our proofs, we give 
a very brief outline here of what it means, and refer to the book of 
Bousfield and Kan \cite{bk} for more details.  They define $p$-completion 
as a functor from spaces to spaces, which we denote $(-)\pcom$ here, and 
this functor comes with a map $X\Right3{\kappa_p(X)}X\pcom$ which is 
natural in $X$.  For any map $f\:X\Right2{}Y$, $f\pcom$ is a homotopy 
equivalence if and only if $f$ is a mod $p$ equivalence; i.e., 
$H^*(f;\F_p)$ is an isomorphism from $H^*(Y;\F_p)$ to $H^*(X;\F_p)$.  A 
space $X$ is called ``$p$-good'' if $\kappa_p(X\pcom)$ is a homotopy 
equivalence (equivalently, $\kappa_p(X)$ is a mod $p$ equivalence).  In 
particular, all spaces with finite fundamental group are $p$-good.  If $X$ 
is $p$-good, then $\kappa_p(X)\:X\Right2{}X\pcom$ is universal among all 
mod $p$ equivalences $X\Right2{}Y$.  If $X$ and $Y$ are both $p$-good, 
then $X\pcom\simeq{}Y\pcom$ (the $p$-completions are homotopy equivalent) 
if and only if there is a third space $Z$, and mod $p$ equivalences 
$X\Right2{}Z\Left2{}Y$.  

By a theorem of Friedlander (stated as Theorem \ref{th:friedl} 
below), $B(^\tau\!\gg(q))\pcom$ is the homotopy fixed space (Definition 
\ref{d:X-ha}) of the action of $\tau\psi^q$ on $B\gg(\C)\pcom$.  Theorem 
\ref{ThA} follows from this together with a general result about homotopy 
fixed spaces (Theorem \ref{t:Xha=Xhb}), which says that under certain 
conditions on a space $X$, two self homotopy equivalences have equivalent 
homotopy fixed sets if they generate the same closed subgroup of the group 
of all self equivalences.

Corresponding results for the Suzuki and Ree groups can also be shown 
using this method of proof.  But since there are much more elementary proofs 
of these results (all equivalences are induced by inclusions of groups), 
and since it seemed difficult to find a nice formulation of the theorem 
which included everything, we decided to leave them out of the statement.

As another application of these results, we prove that for any prime $p$ 
and any prime power $q\equiv1$ (mod $p$), the fusion systems 
$\calf_p(G_2(q))$ and $\calf_p({}^3\!D_4(q))$ are isotypically equivalent 
if $p\ne3$, and the fusion systems $\calf_p(F_4(q))$ and 
$\calf_p({}^2\!E_6(q))$ are isotypically equivalent if $p\ne2$ (Example 
\ref{G2-3D4}).  However, while this provides another example of how our 
methods can be applied, the first equivalence (at least) can also be shown 
by much simpler methods.

Theorem \ref{ThA} is certainly not surprising to the experts, who are 
familiar with it by observation.  It seems likely that it can also be 
shown directly using a purely algebraic proof, but the people we have 
asked do not know of one, and there does not seem to be any in the 
literature.  There is a very closely related result by Michael Larsen 
\cite[Theorem A.12]{GR}, restated below as Theorem \ref{t:gr}.  It implies 
that two Chevalley groups $\gg(K)$ and $\gg(K')$ over algebraically closed 
fields of characteristic prime to $p$ have equivalent $p$-fusion systems 
when defined appropriately for these infinite groups.  There are standard 
methods for comparing the finite subgroups of $\gg(\widebar{\F}_q)$ (of 
order prime to $q$) with those in its finite Chevalley subgroups (see, 
e.g., Proposition \ref{gls3-215}), but we have been unable to get enough 
control over them to prove Theorem \ref{ThA} using Larsen's theorem. 

The paper is organized as follows.  In Section 1, we give a general survey 
of fusion categories of finite groups and their relationship to 
$p$-completed classifying spaces.  Then, in Section 2, we prove a general 
theorem (Theorem \ref{t:Xha=Xhb}) comparing homotopy fixed points of 
different actions on the same space, and apply it in Section 3 to prove 
Theorem \ref{ThA}.  In Section 4, we show a second result about homotopy 
fixed points, which is used to prove the result comparing fusion systems 
of $G_2(q)$ and ${}^3D_4(q)$, and $F_4(q)$ and ${}^2\!E_6(q)$.  We finish 
with a brief sketch in Section 5 of some elementary techniques for proving 
special cases of Theorem \ref{ThA} for some classical groups, and more 
generally a comparison of fusion systems of classical groups at odd primes.

\newcommand{\calg}{\mathcal{G}}
\newcommand{\diag}{\operatorname{\rm diag\,}}
\newcommand{\order}{\operatorname{\rm order}\nolimits}
\def\lcm(#1){\textup{lcm}(#1)}

\section{Fusion categories}
\label{sec:fusdef}

We begin with a quick summary of what is needed here about fusion systems 
of finite groups.  

\begin{defn}  
For any finite group $G$ and any prime $p$, $\calf_p(G)$ denotes the 
category whose objects are the $p$-subgroups of $G$, and where 
	\[ \Mor_{\calf_p(G)}(P,Q) = \{\varphi\in\Hom(P,Q) \,|\, 
	\varphi=c_x \textup{ for some $x\in{}G$} \} ~. \]
Here, $c_x$ denotes the conjugation homomorphism:  $c_x(g)=xgx^{-1}$.  If 
$S\in\sylp{G}$ is a Sylow $p$-subgroup, then 
$\calf_S(G)\subseteq\calf_p(G)$ denotes the full subcategory with objects 
the subgroups of $S$.
\end{defn}

A functor $F\:\calc\Right2{}\calc'$ is an \emph{equivalence} of categories 
if it induces bijections on isomorphism classes of objects and on all 
morphism sets.  This is equivalent to the condition that there be a 
functor from $\calc'$ to $\calc$ such that both composites are naturally 
isomorphic to the identity.  An inclusion of a full subcategory is an 
equivalence if and only if every object in the larger category is 
isomorphic to some object in the smaller one.  Thus when $G$ is finite and 
$S\in\sylp{G}$, the inclusion $\calf_S(G)\subseteq\calf_p(G)$ is an 
equivalence of categories by the Sylow theorems.  

In general, we write $\Psi\:\calc_1\Right2{\cong}\calc_2$ to mean that 
$\Psi$ is an isomorphism of categories (bijective on objects and on 
morphisms); and $\Psi\:\calc_1\Right2{\simeq}\calc_2$ to mean that $\Psi$ 
is an equivalence of categories.

In the following definition, for any finite $G$, $\lambda_G$ denotes the 
forgetful functor from $\calf_p(G)$ to the category of groups.  

\begin{defn}\label{defn:fuspres}
Fix a prime $p$, a pair of finite groups $G$ and $G^*$, and Sylow 
$p$-subgroups $S\in\sylp{G}$ and $S^*\in\sylp{G^*}$.  
\begin{enumerate}\renewcommand{\labelenumi}{\textup{(\alph{enumi})}}
\item An isomorphism $\varphi\:S\Right2{\cong}S^*$ is \emph{fusion 
preserving} if for all $P,Q\le{}S$ and $\alpha\in\Hom(P,Q)$, 
	\[ \alpha\in\Mor_{\calf_p(G)}(P,Q)
	\quad\Longleftrightarrow\quad
	\varphi\alpha\varphi^{-1}\in
	\Mor_{\calf_p(G^*)}(\varphi(P),\varphi(Q)). \]

\item An equivalence of categories $T\:\calf_p(G)\Right2{}\calf_p(G^*)$ 
is \emph{isotypical} if there is a natural isomorphism of functors 
$\omega\:\lambda_G\Right2{}\lambda_{G^*}\circ{}T$; i.e., if there are 
isomorphisms $\omega_P\:P\Right2{\cong}T(P)$ such that 
$\omega_Q\circ\varphi=T(\varphi)\circ\omega_P$ for each 
$\varphi\in\Hom_G(P,Q)$.  
\end{enumerate}
\end{defn}

In other words, in the above situation, an isomorphism 
$\varphi\:S\Right2{}S^*$ is fusion preserving if and only if it induces an 
isomorphism from $\calf_{S}(G)\Right2{\cong}\calf_{S^*}(G^*)$ by sending 
$P$ to $\varphi(P)$ and $\alpha$ to $\varphi\alpha\varphi^{-1}$.  Any such 
isomorphism of categories extends to an equivalence 
$\calf_p(G)\Right2{\simeq}\calf_p(G^*)$, which is easily seen to be 
isotypical.  In fact, two fusion categories $\calf_p(G)$ and 
$\calf_p(G^*)$ are isotypically equivalent if and only if there is a 
fusion preserving isomorphism between Sylow $p$-subgroups, as is shown in 
the following proposition. 

\begin{prop} \label{p:isot<=>fpres}
Fix a pair of finite groups $G$ and $G^*$, a prime $p$ and Sylow 
$p$-subgroups $S\le{}G$ and $S^*\le{}G^*$.  Then the following are 
equivalent:
\begin{enumerate}\renewcommand{\labelenumi}{\textup{(\alph{enumi})}}
\item There is a fusion preserving isomorphism $\varphi\:S\Right2{\cong}S^*$.
\item $\calf_p(G)$ and $\calf_p(G^*)$ are isotypically equivalent.
\item There are bijections $\Rep(P,G)\Right2{\cong}\Rep(P,G^*)$, for all 
finite $p$-groups $P$, which are natural in $P$.
\end{enumerate}
\end{prop}

\begin{proof} This was essentially shown by Martino and Priddy 
\cite{martino-priddy96}, but not completely explicitly.  By the above 
remarks, (a) implies (b).

\smallskip

\noindent\textbf{(b $\Longrightarrow$ c) : } 
Fix an isotypical equivalence 
$T\:\calf_p(G)\Right2{}\calf_p(G^*)$, and let $\omega$ be an associated 
natural isomorphism.  Thus $\omega_P\in\Iso(P,T(P))$ for each $p$-subgroup 
$P\le{}G$, and $\omega_Q\circ\alpha=T(\alpha)\circ\omega_P$ for all 
$\alpha\in\Hom_G(P,Q)$.  For each $p$-group $Q$, $\omega$ defines 
a bijection from $\Hom(Q,G)$ to $\Hom(Q,G^*)$ by sending $\rho$ to 
$\omega_{\rho(Q)}\circ\rho$.  For $\alpha,\beta\in\Hom(Q,G)$, the diagram
	\[ \xymatrix@C+20pt@R+5pt{
	Q \ar[r]^{\alpha} \ar@{=}[d] & \alpha(Q)\ar[r]^{\omega_{\alpha(Q)}} 
	\ar@{-->}[d]^{\gamma} & T(\alpha(Q)) \ar@{-->}[d]^{T(\gamma)} \\
	Q \ar[r]^{\beta} & \beta(Q) \ar[r]^{\omega_{\beta(Q)}} & 
	T(\beta(Q)) \rlap{~,} } \]
together with the fact that $T$ is an equivalence, proves that $\alpha$ and 
$\beta$ are $G$-conjugate (there exists $\gamma$ which makes the left hand 
square commute) if and only if $\omega_{\alpha(Q)}\circ\alpha$ and 
$\omega_{\beta(Q)}\circ\beta$ are $G^*$-conjugate (there exists 
$T(\gamma)$).  Thus $T$ induces bijections 
$\Phi\:\Rep(Q,G)\Right2{\cong}\Rep(Q,G^*)$; and similar arguments show that 
$\Phi$ is natural in $Q$.

\smallskip

\noindent\textbf{(c $\Longrightarrow$ a) : } Fix a natural bijection 
$\Phi\:\Rep(-,G) \xrightarrow{\cong} \Rep(-,G^*)$ of functors on finite 
$p$-groups.  By naturality, $\Phi$ preserves kernels, and hence restricts 
to a bijection between classes of injections.  In particular, there are 
injections of $S$ into $G^*$ and $S^*$ into $G$, and thus $S\cong{}S^*$.  
Since conjugation defines a fusion preserving isomorphism between any two 
Sylow $p$-subgroups of $G^*$, we can assume $S^*=\Im(\Phi_S(\incl_S^G))$.  
Using the naturality of $\Phi$, it is straightforward to check that 
$\Phi_S(\incl_S^G)$ is fusion preserving as an isomorphism from $S$ to 
$S^*$.
\end{proof}

We also note the following, very elementary result about comparing fusion 
systems.

\begin{prop} \label{F(G/Z)}
Fix a prime $p$, and a pair of groups $G_1$ and $G_2$ such that 
$\calf_p(G_1)$ 
is isotypically equivalent to $\calf_p(G_2)$.  Then the following hold, 
where ``$\simeq$'' always means isotypically equivalent.
\begin{enumerate}\renewcommand{\labelenumi}{\textup{(\alph{enumi})}}
\item If $Z_i\le{}Z(G_i)$ is central of order prime to $p$, then 
$\calf_p(G_i/Z_i)\simeq\calf_p(G_i)$.  
\item If $Z_1\le{}Z(G_1)$ is a central $p$-subgroup, and $Z_2\le{}G_2$ is 
its image under some fusion preserving isomorphism between Sylow 
$p$-subgroups of the $G_i$, then $\calf_p(G_1)\simeq\calf_p(C_{G_2}(Z_2))$ 
and $\calf_p(G_1/Z_1)\simeq\calf_p(C_{G_2}(Z_2)/Z_2)$.  
\item $\calf_p([G_1,G_1])\simeq\calf_p([G_2,G_2])$.
\end{enumerate}
\end{prop}

\begin{proof}  Points (a) and (b) are elementary.  To prove (c), first fix 
Sylow subgroups $S_i\in\sylp{G_i}$ and a fusion preserving isomorphism 
$\varphi\:S_1\Right2{\cong}S_2$.  By the focal subgroup theorem (cf.\ 
\cite[Theorem 7.3.4]{Gorenstein}), 
$\varphi(S_1\cap[G_1,G_1])=S_2\cap[G_2,G_2]$.  By \cite[Theorem 
4.4]{BCGLO2}, for each $i=1,2$, there is a unique fusion subsystem ``of 
$p$-power index'' in $\calf_{S_i}(G_i)$ over the focal subgroup 
$S_i\cap[G_i,G_i]$, which must be the fusion system of $[G_i,G_i]$.  Hence 
$\varphi$ restricts to an isomorphism which is fusion preserving with 
respect to the commutator subgroups.  
\end{proof}

Proposition \ref{F(G/Z)} implies, for example, that whenever 
$\calf_p(GL_n(q))\simeq\calf_p(GL_n(q'))$ for $q$ and $q'$ prime to $p$, 
then there are also equivalences $\calf_p(SL_n(q))\simeq\calf_p(SL_n(q'))$,
$\calf_p(PSL_n(q))\simeq\calf_p(PSL_n(q'))$, etc. 

The following theorem of Martino and Priddy shows that the $p$-fusion in a 
finite group is determined by the homotopy type of its $p$-completed 
classifying space.  The converse (the Martino-Priddy conjecture) is also 
true, but the only known proof uses the classification of finite simple 
groups \cite{limz-odd,limz}.

\begin{thm}\label{th:easyMP}
Assume $p$ is a prime, and $G$ and $G'$ are finite groups, such that 
$BG\pcom\simeq{}BG'\pcom$.  Then $\calf_p(G)$ and $\calf_p(G')$ are 
isotypically equivalent. 
\end{thm}

\begin{proof}  This was shown by Martino and Priddy in 
\cite{martino-priddy96}.  The key ingredient in the proof is a theorem of 
Mislin \cite[pp.457--458]{mislin:90}, which says that for any finite 
$p$-group $Q$ and any finite group $G$, there is a bijection
	\[ \Rep(Q,G) \RIGHT4{B\pcom}{\cong} [BQ,BG\pcom]~, \]
where $B\pcom$ sends the class of a homomorphism $\rho\:Q\rTo{}G$ to 
the $p$-completion of $B\rho\:BQ\Right2{}BG$.  
Thus any homotopy equivalence $BG\pcom\Right2{\simeq}BG'\pcom$ induces 
bijections $\Rep(Q,G)\cong\Rep(Q,G')$, for all $p$-groups $Q$, which are 
natural in $Q$.  The theorem now follows from Proposition 
\ref{p:isot<=>fpres}.
\end{proof}

The following proposition will also be useful.  When $H\le{}G$ is a pair of 
groups, we regard $\calf_p(H)$ as a subcategory of $\calf_p(G)$.  

\begin{prop} \label{full-subcat}
If $H\le{}G$ is a pair of groups, then $\calf_p(H)$ is a full subcategory 
of $\calf_p(G)$ if and only if the induced map
	\[ \Rep(P,H) \Right4{} \Rep(P,G) \]
is injective for all finite $p$-groups $P$.
\end{prop}

\begin{proof}  Assume $\Rep(P,H)$ injects into $\Rep(P,G)$ for all $P$.  
For each pair of $p$-subgroups $P,Q\le{}H$ and each 
$\varphi\in\Hom_G(P,Q)$, $[\incl_P^G]=[\incl_Q^G\circ\varphi]$ in 
$\Rep(P,G)$, so $[\incl_P^H]=[\incl_Q^H\circ\varphi]$ in 
$\Rep(P,H)$, and thus $\varphi\in\Hom_H(P,Q)$.  This proves that 
$\calf_p(H)$ is a full subcategory of $\calf_p(G)$.

Conversely, assume $\calf_p(H)$ is a full subcategory.  Fix a finite 
$p$-group $P$, and $\alpha,\beta\in\Hom(P,H)$ such that $[\alpha]=[\beta]$ 
in $\Rep(P,G)$.  Let $\varphi\in\Hom_G(\alpha(P),\beta(P))$ be such 
that $\varphi\circ\alpha=\beta$.  Then $\varphi\in\Hom_H(\alpha(P),\beta(P))$ 
since $\calf_p(H)$ is a full subcategory, and so $[\alpha]=[\beta]$ in 
$\Rep(P,H)$.  This proves injectivity.  
\end{proof}

Our goal in the next three sections is to construct isotypical equivalences 
between fusion systems of finite groups at a prime $p$ by constructing 
homotopy equivalences between their $p$-completed classifying spaces.

\newcommand{\padic}{\Z_p}


\section{Homotopy fixed points of self homotopy equivalences}
\label{sec:htpyfix}

\newcommand{\Tel}{\textup{Tel}}
\newcommand{\invlim}{\setbox1=\hbox{\rm lim}
    \setbox2=\hbox to \wd1{\leftarrowfill} \ht2=0pt \dp2=-1pt
    \mathop{\vtop{\baselineskip=5pt\box1\box2}}}
\newcommand{\Outt}{\textup{Out}}

We start by defining homotopy orbit spaces and homotopy fixed spaces for a 
self homotopy equivalence of a space; i.e., for a homotopy action of the 
group $\Z$.  As usual, $I$ denotes the unit interval $[0,1]$.

\begin{defn} \label{d:X-ha}
Fix a space $X$, and a map $\alpha\:X\Right2{}X$.  
\begin{enumerate}\renewcommand{\labelenumi}{\textup{(\alph{enumi})}}
\item When $\alpha$ is a homeomorphism, the homotopy orbit space 
$X_{h\alpha}$ and homotopy fixed space $X^{h\alpha}$ of $\alpha$ are 
defined as follows:
\begin{itemize}  
\item $X_{h\alpha}=(X\times{}I)/{\sim}$, where $(x,1)\sim(\alpha(x),0)$ 
for all $x\in{}X$.
\item $X^{h\alpha}$ is the space of all continuous maps 
$\gamma\:I\Right2{}X$ such that $\gamma(1)=\alpha(\gamma(0))$.  
\end{itemize}
\item When $\alpha$ is a homotopy equivalence but not a homeomorphism, 
define the \emph{double mapping telescope} of $\alpha$ by setting 
	\[ \Tel(\alpha) = (X\times{}I\times{}\Z)/{\sim} 
	\quad\textup{where}\quad
	(x,1,n)\sim(\alpha(x),0,n+1) \quad \textup{$\forall$ $x\in{}X$, 
	$n\in\Z$.} 
	\]
Let $\widehat{\alpha}\:\Tel(\alpha)\Right2{}\Tel(\alpha)$ be the 
homeomorphism $\widehat{\alpha}([x,t,n])=[x,t,n-1]$.  Then set 
	\[ X_{h\alpha} = \Tel(\alpha)_{h\widehat{\alpha}} 
	\qquad\textup{and}\qquad
	X^{h\alpha} = \Tel(\alpha)^{h\widehat{\alpha}} ~, \]
where $\Tel(\alpha)_{h\widehat{\alpha}}$ and 
$\Tel(\alpha)^{h\widehat{\alpha}}$ are defined as in (a).
\end{enumerate}
\end{defn}

The space $X_{h\alpha}$, when defined as in (a), is also known as the 
\emph{mapping torus} of $\alpha$.  In this situation, $X^{h\alpha}$ is 
clearly the space of sections of the bundle 
$X_{h\alpha}\Right2{p_\alpha}S^1$, defined by identifying $S^1$ with 
$I/(0\sim1)$.  

When $\alpha$ is a homotopy equivalence, the double mapping telescope 
$\Tel(\alpha)$ is homotopy equivalent to $X$.  Thus the idea in part (b) 
of the above definition is to replace $(X,\alpha)$ by a pair 
$(\widehat{X},\widehat{\alpha})$ with the same homotopy type, but such 
that $\widehat{\alpha}$ is a homeomorphism.  The following lemma helps to 
motivate this approach.

\begin{lemma} \label{l:equiv-htyfix}
Fix spaces $X$ and $Y$, a homotopy equivalence $f\:X\Right2{}Y$, and 
homeomorphisms $\alpha\:X\Right2{}X$ and $\beta\:Y\Right2{}Y$ such that 
$\beta\circ{}f\simeq{}f\circ\alpha$.  Then there are homotopy equivalences
	\[ X_{h\alpha}\simeq Y_{h\beta} \qquad\textup{and}\qquad
	X^{h\alpha}\simeq Y^{h\beta}, \]
where these spaces are defined as in Definition \ref{d:X-ha}(a). 
\end{lemma}

\begin{proof}  Fix a homotopy $F\:X\times{}I\Right2{}Y$ such that 
$F(x,0)=f(x)$ and $F(x,1)=\beta^{-1}\circ{}f\circ\alpha$.  Define 
	\[ h\:\underset{=(X\times{}I)/{\sim}}{X_{h\alpha}}\Right5{}
	\underset{=(Y\times{}I)/{\sim}}{Y_{h\beta}} \]
by setting $h(x,t)=(F(x,t),t)$.  If $g$ is a homotopy inverse to $f$ and 
$G$ is a homotopy from $g$ to $\alpha^{-1}\circ{}g\circ\beta$, then these 
define a map from $Y_{h\beta}$ to $X_{h\alpha}$ which is easily seen to be 
a homotopy inverse to $h$.  

We thus have a homotopy equivalence beween $X_{h\alpha}$ and $X_{h\beta}$ 
which commutes with the projections to $S^1$.  Hence the spaces 
$X^{h\alpha}$ and $Y^{h\beta}$ of sections of these bundles are homotopy 
equivalent.
\end{proof}

Lemma \ref{l:equiv-htyfix} also shows that when $\alpha$ is a 
homeomorphism, the two constructions of $X_{h\alpha}$ and $X^{h\alpha}$ 
given in parts (a) and (b) of Definition \ref{d:X-ha} are homotopy 
equivalent.


\begin{rmk} \label{rmk:htyfix}
The homotopy fixed point space $X^{h\alpha}$ of a homeomorphism 
$\alpha$ can also be described as the homotopy pullback of 
the maps
	\[ X \Right4{\Delta} X\times X \Left4{(\Id,\alpha)} X~, \]
where $\Delta$ is the diagonal map $\Delta(x)=(x,x)$.  In other words, 
$X^{h\alpha}$ is the space of triples $(x_1,x_2,\phi)$, where 
$x_1,x_2\in{}X$, and $\phi$ is a path in $X\times{}X$ from 
$\Delta(x_1)=(x_1,x_1)$ to $(x_2,\alpha(x_2))$.  Thus $\phi$ is a pair of 
paths in $X$, one from $x_1$ to $x_2$ and the other from $x_1$ to 
$\alpha(x_2)$, and these two paths can be composed to give a single path 
from $x_2$ to $\alpha(x_2)$ which passes through the (arbitrary) point 
$x_1$.  Hence this definition is equivalent to the one given above.  It 
helps to explain the name ``homotopy fixed point set'', since the ordinary 
pullback of the above maps can be identified with the space of all 
$x\in{}X$ such that $\alpha(x)=x$.
\end{rmk}

For any space $X$, we set $\widehat{H}^i(X;\Z_p)=\invlim H^i(X;\Z/p^k)$ for 
each $i$, and let $\widehat{H}^*(X;\Z_p)$ be the sum of the 
$\widehat{H}^i(X;\Z_p)$.  If $H^*(X;\F_p)$ is finite in each degree, then 
$\widehat{H}^*(X;\Z_p)$ is isomorphic to the usual cohomology ring 
$H^*(X;\Z_p)$ with coefficients in the $p$-adics.

By the $p$-adic topology on $\Outt(X)$, we mean the topology for which 
$\{U_k\}$ is a basis of open neighborhoods of the identity, where 
$U_k\le\Outt(X)$ is the group of automorphisms which induce the identity on 
$H^*(X;\Z/p^k)$.  Thus this topology is Hausdorff if and only if $\Outt(X)$ 
is detected on $\widehat{H}^*(X;\Z_p)$.  

\begin{thm}\label{t:Xha=Xhb}
Fix a prime $p$.  Let $X$ be a connected, $p$-complete space such that
\begin{itemize}  
\item $H^*(X;\F_p)$ is noetherian, and 
\item $\Outt(X)$ is detected on $\widehat{H}^*(X;\Z_p)$.
\end{itemize}
Let $\alpha$ and $\beta$ be self homotopy equivalences of $X$ which 
generate the same closed subgroup of $\Outt(X)$ under the $p$-adic 
topology.  Then $X^{h\alpha}\simeq X^{h\beta}$.
\end{thm}

\begin{proof}  Upon replacing $X$ by the double mapping telescope of 
$\alpha$, we can assume that $\alpha$ is a homeomorphism.  By Lemma 
\ref{l:equiv-htyfix}, this does not change the homotopy type of 
$X^{h\alpha}$ or $X^{h\beta}$.

Let $r\ge1$ be the smallest integer prime to $p$ such that 
the action of $\alpha^r$ on $H^*(X;\F_p)$ has $p$-power order.  (The 
action of $\alpha$ on the noetherian ring $H^*(X;\F_p)$ has finite order.) 
Since $\widebar{\gen{\alpha}}=\widebar{\gen{\beta}}$, $H^*(\alpha;\F_p)$ 
and $H^*(\beta;\F_p)$ generate the same subgroup in $\Aut(H^*(X;\F_p))$, 
hence have the same order, and so $r$ is also the smallest integer prime 
to $p$ such that the action of $\beta^r$ on $H^*(X;\F_p)$ has $p$-power 
order. 

Let 
	\[ p_\alpha\:X_{h\alpha}\Right4{}S^1 \qquad\textup{and}\qquad
	p_\beta\:X_{h\beta}\Right4{}S^1 \]
be the canonical fibrations.  Let 
	\[ \widetilde{p}_\alpha\:\widetilde{X}_{h\alpha}\Right4{}S^1 
	\qquad\textup{and}\qquad
	\widetilde{p}_\beta\:\widetilde{X}_{h\beta}\Right4{}S^1 \]
be their $r$-fold cyclic covers, considered as equivariant maps between 
spaces with $\Z/r$-action.  Thus 
$\widetilde{X}_\alpha\simeq{}X_{h\alpha^r}$ (and $\widetilde{p}_\alpha$ is 
its canonical fibration), and the same for $\widetilde{X}_\beta$.  Also, 
since each section of $p_\alpha$ lifts to a unique equivariant section of 
$\widetilde{p}_\alpha$, and each equivariant section factors through as 
section of $p_\alpha$ by taking the orbit map, $X^{h\alpha}$ is the space 
of all equivariant sections of $\widetilde{p}_\alpha$. 

Since $\alpha^r$ acts on $H^*(X;\F_p)$ with order a power of $p$, this 
action is nilpotent; and by \cite[II.5.1]{bk}, the homotopy fiber of 
$\widetilde{p}_\alpha\pcom$ has the homotopy type of $X\pcom\simeq{}X$.  
Thus the rows in the following diagram are (homotopy) fibration sequences:
	\begin{diagram}[w=35pt]
	X & \rTo & \widetilde{X}_{h\alpha} & \rTo^{\widetilde{p}_\alpha} 
	& S^1 \\
	\dIgual && \dTo>{\kappa_p} && \dTo>{\kappa_p} \\
	X & \rTo & (\widetilde{X}_{h\alpha})\pcom & 
	\rTo^{\widetilde{p}_\alpha\pcom} & S^1\pcom \rlap{~,}
	\end{diagram}
and so the right hand square is a homotopy pullback.  By the definition of 
$p$-completion in \cite{bk}, the induced actions of $\Z/r$ on $S^1\pcom$ 
and on $(\widetilde{X}_{h\alpha})\pcom$ are free, since the actions on the 
uncompleted spaces are free.  Hence $X^{h\alpha}$ can be described (up to 
homotopy), not only as the space of $\Z/r$-equivariant sections of 
$\widetilde{p}_\alpha$, but also as the space of $\Z/r$-equivariant 
liftings of $\kappa_p(S^1)\:S^1\Right2{}S^1\pcom$ along 
$\widetilde{p}_\alpha\pcom$.  In other words, 
	\beqq X^{h\alpha} \simeq \textup{fiber}\bigl( 
	\map_{\Z/r}(S^1,(\widetilde{X}_{h\alpha})\pcom) 
	\Right5{\widetilde{p}_\alpha\pcom\circ-} 
	\map_{\Z/r}(S^1,S^1\pcom) \bigr) \label{e:lift}
	\eeqq
(the fiber over $\kappa_p(S_1)$).

Consider again the $p$-completed fibration sequence 
	\[ X \Right5{} (\widetilde{X}_{h\alpha})\pcom 
	\Right5{\widetilde{p}_\alpha} BS^1\pcom~, \]
and its orbit fibration
	\[ X \Right5{} (\widetilde{X}_{h\alpha})\pcom/(\Z/r)
	\Right5{\widehat{p}_\alpha} BS^1\pcom/(\Z/r)~. \]
Here, $\pi_1(BS^1\pcom)\cong\Z_p$, and 
$\pi_1(BS^1\pcom/(\Z/r))\cong\Z_p\times\Z/r$ (the completion of $\Z$ with 
respect to the ideals $rp^i\Z$).  Since $\alpha^r$ acts on 
$H^*(X;\Z/p)$ with order a power of $p$, it also acts on each 
$H^*(X;\Z/p^k)$ with order a power of $p$, and hence the homotopy action 
of $\pi_1(BS^1\pcom)$ on $X$ has as image the $p$-adic closure of 
$\gen{\alpha^r}$.  Thus the homotopy action of $\pi_1(BS^1\pcom/(\Z/r))$ 
on $X$, defined by the fibration $\widehat{p}_\alpha$, has as image the 
$p$-adic closure of $\gen{\alpha}$.  

Since $\beta\in\widebar{\gen{\alpha}}$ by assumption, we can represent 
it by a map $S^1\Right2{b}S^1\pcom/(\Z/r)$.  Define $Y$ to be the 
homotopy pullback defined by the following diagram:
	\begin{diagram}[w=45pt]
	X & \rTo & Y & \rTo^{p'} & S^1 \\
	\dIgual && \dTo>{f} && \dTo>b \\
	X & \rTo & (\widetilde{X}_{h\alpha})\pcom/(\Z/r) & 
	\rTo^{\widehat{p}_\alpha} 
	& S^1\pcom/(\Z/r) \rlap{~.}
	\end{diagram}
Thus the canonical generator of $\pi_1(S^1)$ induces $\beta\in\Outt(X)$, 
and so $(Y,p')\simeq{}(X_{h\beta},p_\beta)$.  Upon taking $r$-fold covers 
and then completing the first row, this induces a map of fibrations
	\begin{diagram}[w=35pt]
	X & \rTo & (\widetilde{X}_{h\beta})\pcom & 
	\rTo^{\widetilde{p}_\beta} & S^1\pcom \\
	\dIgual && \dTo>{f}<{\simeq} && \dTo>{\widetilde{b}} \\
	X & \rTo & (\widetilde{X}_{h\alpha})\pcom & 
	\rTo^{\widetilde{p}_\alpha} 
	& S^1\pcom \rlap{~,}
	\end{diagram}
which is an equivalence since $\widetilde{b}$ is an equivalence (since 
$\gen{\beta}$ is dense in $\widebar{\gen{\alpha}}$).  Also, 
$f$ and $\widetilde{b}$ are equivariant with respect to some 
automorphism of $\Z/r$.

The maps $S^1 \Right1{} S^1\pcom  \Left1{} (X)_{h\Z_p} $ determine a 
commutative diagram
   \begin{equation}\label{diag.1}
	\begin{diagram}[w=60pt]
	\map_{\Z/r}(S^1\pcom,(\widetilde{X}_{h\alpha})\pcom) & \rTo &  
	\map_{\Z/r}(S^1,(\widetilde{X}_{h\alpha})\pcom) \\
	\dTo>t && \dTo>u \\
	\map_{\Z/r}(S^1\pcom,S^1\pcom) & \rTo & \map_{\Z/r}(S^1,S^1\pcom) 
	\end{diagram}
   \end{equation}
which in turn induces a map between the respective fibres.  The horizontal 
arrows in \eqref{diag.1} are homotopy equivalences because the target 
spaces in the respective mapping spaces are $p$-complete and $\Z/r$ acts 
freely on the source spaces.  Hence the 
fibers of the vertical maps in \eqref{diag.1} are homotopy equivalent.  
Since $u$ has fiber $X^{h\alpha}$ by \eqref{e:lift}, this proves that 
$X^{h\alpha}$ has the homotopy type of the space of equivariant sections 
of the bundle $(X_{h\alpha})\pcom \Right2{\widetilde{p}_\alpha\pcom} 
S^1\pcom$.  Since this bundle is equivariantly equivalent to the one 
with total space $(\widetilde{X}_{h\beta})\pcom$, the same argument 
applied to $\beta$ proves that $X^{h\alpha}\simeq{}X^{h\beta}$.
\end{proof}

Our main application of Theorem \ref{t:Xha=Xhb} is to the case where 
$X=BG\pcom$ for a compact connected Lie group $G$.  

\begin{cor}\label{prop:taupsiq}
Let $G$ be a compact connected Lie group,  and let $\alpha,\beta \in 
\Out{}{BG\pcom}$ be two self equivalences of the $p$-completed classifying 
space.  If $\alpha$ and $\beta$ generate the same closed subgroup of 
$\Out{}{BG\pcom}$, then $(BG\pcom)^{h\alpha} \simeq (BG\pcom)^{h\beta}$. 
\end{cor}

\begin{proof}  From the spectral sequence for the fibration 
$U(n)/G\Right2{}BG\Right2{}BU(n)$ for any embedding $G\le{}U(n)$, we see 
that $H^*(BG;\F_p)$ is noetherian.  By \cite[Theorem 2.5]{jmo:selfho}, 
$\Outt(BG\pcom)$ is detected by its restriction to $BT\pcom$ for a maximal 
torus $T$, and hence by invariant theory is detected by 
$\Q\otimes_{\Z}\widehat{H}^*(BG;\Z_p)$.  So the hypotheses of Theorem 
\ref{t:Xha=Xhb} hold when $X=BG\pcom$.  
\end{proof}

The hypotheses on $X$ in Theorem \ref{t:Xha=Xhb} also apply whenever $X$ 
is the classifying space of a connected $p$-compact group.  The condition 
on cohomology holds by \cite[Theorem 2.3]{dw:fixpt}.  Automorphisms are 
detected by restriction to the maximal torus by \cite[Theorem 1.1]{AGMV} 
(when $p$ is odd) and \cite[Theorem 1.1]{Moeller} (when $p=2$).  


\newcommand{\ttt}{\mathbb{T}}

\section{Finite groups of Lie type}
\label{sec:lietype}

We first fix our terminology.  Let $\gg$ be a connected reductive integral 
group scheme.  Thus for each algebraically closed field $K$, $\gg(K)$ is a 
complex connected algebraic group such that for some finite central 
subgroup $Z\le{}Z(\gg(K))$, $\gg(K)/Z$ is the product of a $K$-torus 
and a semisimple group.  For any prime power $q$, we let $\gg(q)$ be the 
fixed subgroup of the field automorphism $\psi^q$.  Also, if $\tau$ is any 
automorphism of $\gg$ of finite order, then ${}^\tau\!\gg(q)$ will denote 
the fixed subgroup of the composite $\tau\psi^q$ on $\gg(\fqbar)$.  

Note that with this definition, when $\gg=PSL_n$, $\gg(q)$ does not mean 
$PSL_n(q)$ in the usual sense, but rather its extension by diagonal 
automorphisms (i.e., $PGL_n(q)$).  By Proposition \ref{F(G/Z)}, however, 
any equivalence between fusion systems over groups $SL_n(-)$ will also 
induce an equivalence between fusion systems over $PSL_n(-)$.  Also, we 
are not including the case $\gg=SO_n$ for $n$ even, since 
$SO_n(\widebar{\F}_2)$ is not connected.  Instead, when working with 
orthogonal groups in even dimension, we take $\gg=\Omega_n$ (and 
$\Omega_n(K)=SO_n(K)$ when $K$ is algebraically closed of characteristic 
different from two). 

The results in this section are based on Corollary \ref{prop:taupsiq}, 
together with the following theorem of Friedlander.  
Following the terminology of \cite{GLS3}, we define a 
\emph{Steinberg endomorphism} of an algebraic group $\widebar{G}$ over an 
algebraically closed field to be an algebraic endomorphism 
$\psi\:\widebar{G}\Right2{}\widebar{G}$ which is bijective, and whose 
fixed subgroup is finite.  
For any connected complex Lie group $\gg(\C)$ with maximal torus 
$\ttt(\C)$, any prime $p$, and any $m\in\Z$ prime to $p$, 
$\Psi^m\:B\gg(\C)\pcom\Right2{}B\gg(\C)\pcom$ denotes a self equivalence 
whose restriction to $B\ttt(\C)\pcom$ is induced by $(x\mapsto{}x^m)$ (an 
``unstable Adams operation'').  Such a map is unique up to homotopy by 
\cite[Theorem 2.5]{jmo:selfho} (applied to $BG\simeq{}B\gg(\C)$, where $G$ 
is a maximal compact subgroup of $\gg(\C)$). 

\begin{thm}\label{th:friedl}
Fix a connected reductive group scheme $\gg$, a prime power $q$, and a 
prime $p$ which does not divide $q$.  Then for any Steinberg endomorphism 
$\psi$ of $\gg(\widebar{\F}_q)$ with fixed subgroup $H$, 
	\[ BH\pcom \simeq (B\gg(\C)\pcom)^{h\Psi} \]
for some $\Psi\:B\gg(\C)\pcom\Right2{\simeq}B\gg(\C)\pcom$.  If 
$\psi=\tau(\widebar{\F}_q)\circ\gg(\psi^q)$, where $\tau\in\Aut(\gg)$ and 
$\psi^q\in\Aut(\widebar{\F}_q)$ is the automorphism 
$(x\mapsto{}x^q)$, then $\Psi\simeq{}B\tau(\C)\circ\Psi^q$ where $\Psi^q$ 
is as described above.
\end{thm}

\begin{proof}  By \cite[Theorem 12.2]{friedlander82}, $BH\pcom$ is 
homotopy equivalent to a homotopy pullback of maps 
	\[ B\gg(\C)\pcom \Right6{(\Id,\Psi)} 
	B\gg(\C)\pcom\times{}B\gg(\widebar{\F}_q)\pcom 
	\Left6{(\Id,\Id)} B\gg(\C)\pcom \]
for some $\Psi$; and hence is homotopy equivalent to 
$(B\gg(\C)\pcom)^{h\Psi}$ by Remark \ref{rmk:htyfix}.  From the proof of 
Friedlander's theorem, one sees that $\Psi$ is induced by $B\psi$, 
together with the homotopy equivalence $B\gg(\C)\pcom\simeq 
\textup{holim}\bigl((B\gg(\widebar{\F}_q)_{\textup{et}})\pcom\bigr)$ of 
\cite[Proposition 8.8]{friedlander82}.  This equivalence is natural with 
respect to the inclusion of a maximal torus $\ttt$ in $\gg$.  Hence when 
$\psi=\tau(\widebar{\F}_q)\circ\gg(\psi^q)$, $\Psi$ restricts to the 
action on $B\ttt(\C)\pcom$ induced by $\tau$ and $(x\mapsto{}x^q)$.
\end{proof}

Theorem \ref{th:friedl} can now be combined with Corollary 
\ref{prop:taupsiq} to prove Theorem A; i.e., to compare fusion systems 
over different Chevalley groups associated to the same connected group 
scheme $\gg$.  This will be done in the next two propositions.

\begin{prop}\label{pr:G(q)=G(q')}
Fix a prime $p$, a connected reductive integral group scheme $\gg$, and an 
automorphism $\tau$ of $\gg$ of finite order $k$.  Assume, for each $m$ 
prime to $k$, that $\tau^m$ is conjugate to $\tau$ in the group of all 
automorphisms of $\gg$.  Let $q$ and $q'$ be prime powers prime to $p$.  
Assume either 
\begin{enumerate}\renewcommand{\labelenumi}{\textup{(\alph{enumi})}}
\item  $\widebar{\gen{q}}=\widebar{\gen{q'}}$ as subgroups of 
$\Z_p^\times$; or
\item  there is some $\psi^{-1}$ in the Weyl group of $\gg$ which inverts 
all elements of the maximal torus, and 
$\widebar{\gen{-1,q}}=\widebar{\gen{-1,q'}}$ as subgroups of $\Z_p^\times$.
\end{enumerate}
Then there is an isotypical equivalence 
$\calf_p(^\tau\!\gg(q))\simeq\calf_p(^\tau\!\gg(q'))$.  
\end{prop}

\begin{proof}  By \cite[Theorem 2.5]{jmo:selfho}, the group 
$\Outt(B\gg(\C)\pcom)$ is detected by restricting maps to a maximal torus. 
(Every class in this group is represented by some map which sends 
$BT\pcom$ to itself for some maximal torus $T$.)  Hence 
$\Psi^q,\Psi^{q'}\in\Outt(B\gg(\C)\pcom)$ (the maps whose restrictions to 
the maximal torus are induced by $(x\mapsto{}x^q)$ and 
$(x\mapsto{}x^{q'})$, respectively) generate the same closed subgroup of 
$\Outt(B\gg(\C)\pcom)$ whenever $\widebar{\gen{q}}=\widebar{\gen{q'}}$.  


Since $\tau$ is an automorphism of $\gg$, its actions on $\gg(\fqbar)$ and 
$\gg(\widebar{\F}_{q'})$ commute with the field automorphisms $\psi^q$ and 
$\psi^{q'}$.  Thus $(\tau\psi^q)^k=\psi^{q^k}$ has finite fixed subgroup, 
so $\tau\psi^q$ also has finite fixed subgroup, and similarly for 
$\tau\psi^{q'}$.  So by Theorem \ref{th:friedl}, 
	\[ B(^\tau\!\gg(q))\pcom\simeq
	(B\gg(\C)\pcom)^{h(B\tau\circ\Psi^q)}
	\quad\textup{and}\quad
	B(^\tau\!\gg(q'))\pcom\simeq
	(B\gg(\C)\pcom)^{h(B\tau\circ\Psi^{q'})}. \]

Assume $\widebar{\gen{q}}=\widebar{\gen{q'}}$.  Then for some $m$ prime to 
$k=|\tau|$, $q\equiv(q')^m$ modulo 
$\widebar{\gen{q^k}}=\widebar{\gen{q'{}^k}}$.  Hence $B\tau^m\circ\Psi^q$ 
and $B\tau\circ\Psi^{q'}$ generate the same closed subgroup of 
$\Outt(\gg(\C)\pcom)$ under the $p$-adic topology, since they generate the 
same subgroup modulo $\widebar{\gen{\Psi^{q^k}}}$.  Thus
	\[ (B\gg(\C)\pcom)^{h(B\tau\circ\Psi^{q'})} \simeq 
	(B\gg(\C)\pcom)^{h(B\tau^m\circ\Psi^{q})} \simeq 
	(B\gg(\C)\pcom)^{h(B\tau\circ\Psi^{q})} ~, \]
where the first equivalence holds by Corollary \ref{prop:taupsiq}, and the 
second since $\tau$ and $\tau^m$ are conjugate in the group of all 
automorphisms of $\gg$.  So
$B(^\tau\!\gg(q))\pcom\simeq{}B(^\tau\!\gg(q'))\pcom$, and there is an 
isotypical equivalence between the fusion systems of these groups by 
Theorem \ref{th:easyMP}.  

If $-\Id$ is in the Weyl group, then we can regard this as an inner 
automorphism of $\gg(\C)$ which inverts all elements in a maximal torus.  
Thus by \cite{jmo:selfho} again, $\Psi^{-1}\simeq\Id$ in this case.  So by 
the same argument as that just given, 
$\widebar{\gen{-1,q}}=\widebar{\gen{-1,q'}}$ implies 
$(B\gg(\C)\pcom)^{h(B\tau\circ\Psi^{q})}\simeq 
(B\gg(\C)\pcom)^{h(B\tau\circ\Psi^{q'})}$, and hence 
$\calf_p({}^{\tau}\gg(q))\simeq\calf_p({}^{\tau}\gg(q'))$.  
\end{proof}

To make the condition $\widebar{\gen{q}}=\widebar{\gen{q'}}$ more 
concrete, note that for any prime $p$, and any $q,q'$ prime to $p$ of 
order $s$ and $s'$ in $\F_p^\times$, respectively, 
	\[ \widebar{\gen{q}}=\widebar{\gen{q'}} \quad \iff \quad 
	\begin{cases}  
	s=s' \textup{ and } 
	v_p(q^s-1)=v_p(q'{}^s-1) & \textup{if $p$ is odd} \\
	q\equiv{}q' \pmod{8} \textup{ and } 
	v_p(q^2-1)=v_p(q'{}^2-1) & \textup{if $p=2$~.} 
	\end{cases} \]

Corollary \ref{prop:taupsiq} can also be applied to compare fusion systems 
of Steinberg groups with those of related Chevalley groups.  

\begin{prop}\label{pr:G+(q)=G-(q')}
Fix a prime $p$, and a pair $q,q'$ of prime powers prime to $p$ such that 
$\widebar{\gen{-q}}=\widebar{\gen{q'}}$ as subgroups of $\Z_p^\times$.  
Then there are isotypical equivalences:
\begin{enumerate}\renewcommand{\labelenumi}{\textup{(\alph{enumi})}}
\item $\calf_p(SU_n(q))\simeq\calf_p(SL_n(q'))$ for all $n$.
\item $\calf_p(\Spin_{2n}^-(q))\simeq\calf_p(\Spin_{2n}^+(q'))$ for all odd 
$n$.
\item $\calf_p({}^2E_6(q))\simeq\calf_p(E_6(q'))$.
\end{enumerate}
\end{prop}

\begin{proof}  Set $\gg=SL_n$, $\Spin_{2n}$ for $n$ odd, or the simply 
connected $E_6$; and let $\tau$ be the graph automorphism of order two.  
In all of these cases, $\tau$ acts by inverting the elements of some 
maximal torus.  Hence by Theorem \ref{th:friedl}, 
$B(^\tau\!\gg(q))\pcom\simeq(B\gg(\C)\pcom)^{h\Psi^{-q}}$ and 
$B\gg(q')\pcom\simeq(B\gg\pcom)^{h\Psi^{q'}}$.  So 
$B\gg(q)\pcom\simeq{}B\gg(q')\pcom$ by Corollary \ref{prop:taupsiq}, 
\cite[Theorem 2.5]{jmo:selfho}, and the assumption 
$\widebar{\gen{-q}}=\widebar{\gen{q'}}$; and there is an isotypical 
equivalence between the fusion systems of these groups by Theorem 
\ref{th:easyMP}.
\end{proof}

Upon combining this with Proposition \ref{F(G/Z)}, one gets similar 
results for $PSU_n(q)$, $\Omega_{2n}^-(q)$, $P\Omega_{2n}^-(q)$, etc.  

This finishes the proof of Theorem \ref{ThA}.  We now finish the section 
with some remarks about a possible algebraic proof of this result.  The 
following theorem of Michael Larsen in his appendix to \cite{GR} implies 
roughly that two Chevalley groups $\gg(K)$ and $\gg(K')$ over algebraically 
closed fields $K$ and $K'$ have equivalent $p$-fusion systems 
(appropriately defined) for $p$ different from the characteristics of $K$ 
and $K'$.

\begin{thm} \label{t:gr}
Fix a connected group scheme $\gg$, and let $K$ and $K'$ be two algebraically 
closed fields.  Then there are bijections 
	\[ \nu_P\: \Rep(P,\gg(K)) \Right5{\cong} \Rep(P,\gg(K')), \]
for all finite groups $P$ of order prime to $\chr(K)$ and $\chr(K')$, and 
which are natural with respect to $P$, and also with respect to 
automorphisms of $\gg$.  
\end{thm}

\begin{proof} Except for the statement of naturality, this is 
\cite[Theorem A.12]{GR}.  For fields of the same characteristic, the 
bijection is induced by the inclusions of the algebraic closures of their 
prime subfields (\cite[Lemma A.11]{GR}).  When $K=\fqbar$ for a prime $q$, 
$W(K)$ is its ring of Witt vectors (the extension of $\Z_q$ by all roots 
of unity of order prime to $q$), and $K'$ is the algebraic closure of 
$W(K)$, then the bijections 
$\Rep(P,\gg(K))\cong\Rep(P,\gg(W(K)))\cong\Rep(P,\gg(K'))$ are induced by 
the projection $W(K)\Onto2{}K$ and inclusion $W(K)\subseteq{}K'$ in the 
obvious way.  All of these are natural in $P$, and also commute with 
automorphisms of $\gg$.
\end{proof}

The next proposition describes how to compare 
$\Rep(P,\gg(\widebar{\F}_q))$ to the Chevalley and Steinberg groups over 
$\F_q$.  

\begin{prop} \label{gls3-215}
Fix a connected algebraic group $\widebar{G}$ over $\fqbar$ for some $q$, 
and $\sigma$ be a Steinberg endomorphism of $\widebar{G}$, and set 
$G=C_{\widebar{G}}(\sigma)$.  Let $P$ be a finite group, and consider the 
map of sets
	\[ \rho\: \Rep(P,G) \Right5{} \Rep(P,\widebar{G}). \]
Fix $\varphi\in\Hom(P,\widebar{G})$, and let $[\varphi]$ be its class in 
$\Rep(P,\widebar{G})$.  Then $[\varphi]\in\Im(\rho)$ if and only if 
$[\varphi]$ is fixed under the action of $\sigma$ on $\Rep(P,\widebar{G})$.  
When $\varphi(P)\le{}G$, set
	\[ C=C_{\widebar{G}}(P), \quad 
	N=N_{\widebar{G}}(P), \quad \widehat{C}=\pi_0(C), \]
and let $g\in{}N$ act on $\widebar{C}$ by sending $x$ to 
$gx\sigma(g)^{-1}$.  Then there is a bijection 
	\[ B\: \rho^{-1}([\varphi]) \Right5{\cong} \widehat{C}/C, \] 
where $B([c_y\circ\varphi])=y^{-1}\sigma(y)$ for any $y\in\widebar{G}$ 
such that $y\varphi(P)y^{-1}\le{}G$.  Also, $\Aut_G(\varphi(P))$ is the 
stabilizer, under the action of 
$\Aut_{\widebar{G}}(\varphi(P))\cong{}N/C$, of the class of the identity 
element in $\widehat{C}$.
\end{prop}

\begin{proof}  This is a special case of \cite[Theorem 2.1.5]{GLS3}, when 
applied to the set $\Omega$ of all homomorphisms 
$\varphi'\:P\Right2{}\widebar{G}$ which are $\widebar{G}$-conjugate to 
$\varphi$.  
\end{proof}

We assumed at first that an algebraic proof of Theorem \ref{ThA} could 
easily be constructed by applying Theorem \ref{t:gr} and Proposition 
\ref{gls3-215}.  But so far, we have been unable to do so, nor do we know 
of any other algebraic proof of this result.


\section{Homotopy fixed points of proxy actions}
\label{s:proxy}

To get more results of this type, we need to look at more general types of 
actions and their homotopy fixed points.  The concept of ``proxy actions'' 
is due to Dwyer and Wilkerson.

\begin{defn} \label{d:X-hG}
For any discrete group $G$ and any space $X$, a \emph{proxy 
action} of $G$ on $X$ is a fibration $f\:X_{hG}\Right2{}BG$ with fiber 
$X$.  An equivalence of proxy actions $f\:X_{hG}\Right2{}BG$ on $X$ and 
$f'\:Y_{hG}\Right2{}BG$ on $Y$ is a homotopy equivalence 
$\alpha\:X_{hG}\Right2{}Y_{hG}$ such that $f'\circ\alpha=f$.  The 
\emph{homotopy fixed space} $X^{hG}$ of a proxy action 
$f\:X_{hG}\Right2{}BG$ on $X$ is the space of sections 
$s\:BG\Right2{}X_{hG}$ of the fibration $f$.
\end{defn}

Any (genuine) action of $G$ on $X$ can be regarded as a proxy action via 
the Borel construction:  $X_{hG}=EG\times_GX$ is the orbit space of the 
diagonal $G$-action on $EG\times{}X$.  In this case, we can identify
	\[ X^{hG} = \map_G(EG,X) \]
via covering space theory.  

If $\alpha\:X\Right2{}X$ is a homeomorphism, regarded as a $\Z$-action, 
then the mapping torus 
	\[ X_{h\alpha}=\bigl(X\times{}I\bigr)\big/
	\bigl((x,1)\sim(\alpha(x),0)\bigr)~, \]
as defined in Definition \ref{d:X-ha}, is homeomorphic to the Borel 
construction $E\Z\times_{\Z}X$.  So in this case, the homotopy fixed set 
$X^{h\Z}$ of Definition \ref{d:X-hG} is the same as the space 
$X^{h\alpha}$ of Definition \ref{d:X-ha}.  If $\alpha$ is a self homotopy 
equivalence of $X$, then the map from the mapping torus to $S^1=B\Z$ need 
not be a fibration, which is why we need to first replace $X$ by the double 
telescope before defining the homotopy fixed set.  

Now assume, furthermore, that $X$ is $p$-complete, and that the action of 
$\alpha$ on $H^n(X;\F_p)$ is nilpotent for each $n$.  Then by 
\cite[Lemma II.5.1]{bk}, the homotopy fiber of the $p$-completion 
$(X_{h\alpha})\pcom\Right2{}S^1\pcom$ has the homotopy type of 
$X\pcom\simeq{}X$.  Also, $S^1\pcom\simeq{}B\Z_p$, and so this defines a 
proxy action of the $p$-adics on $X$.  By the arguments used in the proof of 
Theorem \ref{t:Xha=Xhb}, the homotopy fixed space of this action has the 
homotopy type of $X^{h\alpha}$.  

Some of the basic properties of proxy actions and their homotopy fixed 
spaces are listed in the following proposition. 

\begin{prop} \label{p:X-hG-props}
Fix a proxy action $X_{hG}\Right2{f}BG$ of a discrete group 
$G$ on a space $X$.
\begin{enumerate}\renewcommand{\labelenumi}{\textup{(\alph{enumi})}}
\item If $Y_{hG}\Right2{f'}BG$ is a proxy action of the same group $G$ on 
another space $Y$, and $\varphi\:X_{hG}\Right2{}Y_{hG}$ is a homotopy 
equivalence such that $f'\circ\varphi=f$, then $\varphi$ induces a 
homotopy equivalence $X^{hG}\simeq{}Y^{hG}$.
\item Let $\widetilde{X}$ be the pullback of $X_{hG}$ and $EG$ over $BG$.  
Then $G$ has a genuine action on $\widetilde{X}$, which as a proxy action, 
is equivalent to that on $X$.  
\end{enumerate}
\end{prop}

\begin{proof}  Point (a) follows easily from the definition.  

In the situation of (b), the 
action of $G$ on $EG$ induces a free action on $\widetilde{X}$.  Consider 
the two $G$-maps 
	\[ \proj_1,\widetilde{f}\circ\proj_2\: EG\times_G\widetilde{X} 
	\Right5{} EG~, \]
where $\widetilde{f}\:\widetilde{X}\Right2{}EG$ is the map coming from the 
pullback square used to define $\widetilde{X}$.  By the universality of 
$EG$, all maps from $EG\times_G\widetilde{X}$ to $EG$ are equivariantly 
homotopic (cf.\ \cite[Theorem 4.12.4]{Husemoller}, and recall that these 
spaces are the total spaces of principal $G$-bundles).  So upon passing to 
the orbit space, the composite 
	\[ EG\times_G\widetilde{X} \Right4{\proj_2/G} X_{hG} 
	\RIGHT4{\widetilde{f}/G}{=f} BG \]
is homotopic to $\proj_1/G$, the map which defines the proxy action of $G$ 
induced by the action on $\widetilde{X}$.  By the homotopy lifting property 
for the fibration $f$, $\proj_2/G$ is homotopic to a map $\alpha$ such that 
$f\circ\alpha=\proj_1/G$, and this is an equivalence between the proxy 
actions.
\end{proof}

If $X_{hG}\Right2{f}BG$ is a proxy action of $G$ on $X$, and $H\le{}G$ is a 
subgroup, then the pullback of $BH$ and $X_{hG}$ over $BG$ defines a proxy 
action $X_{hH}\Right2{f|H}BH$ of $H$ on $X$.  
The first statement in the following proposition is due to Dwyer and 
Wilkerson \cite[10.5]{dw:fixpt}.  

\begin{prop}\label{ZtoZpZr}
Let $f\:X_{hG}\Right2{}BG$ be a proxy action of $G$ on $X$, and let $H$ be 
a normal subgroup of $G$ with quotient group $\pi= G/H$. Then the 
following hold.
\begin{enumerate}\renewcommand{\labelenumi}{\textup{(\alph{enumi})}}
\item There is a proxy action of $\pi$ on $X^{hH}$ with $X^{hG} \simeq 
(X^{hH})^{h\pi}$.
\item Assume $G_0\le{}G$ and $H_0=H\cap G$ are such that the inclusion 
induces an isomorphism $G_0/H_0\Right2{\cong}G/H=\pi$.  Assume also that 
the natural map $X^{hH}\Right2{}X^{hH_0}$ induced by restricting the 
action of $G$ is a homotopy equivalence.  Then $X^{hG}\Right2{}X^{hG_0}$ 
is also a homotopy equivalence.
\end{enumerate}
\end{prop}

\begin{proof} We give an argument for (a) that will be useful in the proof 
of (b). Write $\widehat{B}H=EG/H\simeq{}BH$ and $\widehat{B}G=EG\times_G 
E\pi\simeq{}BG$, 
where $EG\times_G E\pi$ is the orbit space of the diagonal action of $G$ 
on $EG\times E\pi$.  Let 
	\[ \gamma\:BG\Right4{\simeq}\widehat{B}G 
	\qquad\textup{and}\qquad
	\widehat{B}\iota\:\widehat{B}H\Right4{}\widehat{B}G \]
be induced by the diagonal map $EG\Right2{}EG\times{}E\pi$ and its 
composite with the inclusion $E\iota\:EH\Right2{}EG$.  The adjoint of the 
projection 
	\[ \widehat{B}H\times E\pi = EG\times_H E\pi \Right5{} 
	EG\times_G E\pi=\widehat{B}G \]
is a $\pi$-equivariant map 
$E\pi\Right2{}\map(\widehat{B}H,\widehat{B}G)_{\widehat{B}\iota}$, where 
$\map(\widehat{B}H,\widehat{B}G)_{\widehat{B}\iota}$ is the space of all maps 
homotopic to $\widehat{B}\iota$.  
Consider the following homotopy pullback diagram of spaces with $\pi$-action: 
	\begin{diagram}[w=40pt]
	Y & \rTo & \map(EG/H,X_{hG})_{[\widehat{B}\iota]} \\
	\dTo && \dTo>{f'\circ-} \\
	E\pi & \rTo & \map(EG/H,\widehat{B}G)_{\widehat{B}\iota} \rlap{~.} 
	\end{diagram}
Here, $f'=\gamma\circ{}f\:X_{hG}\Right2{}\widehat{B}G$, and 
$\map(-,-)_{[\widehat{B}\iota]}$ means the space of all maps whose 
composite with $f'$ is homotopic to $\widehat{B}\iota$.  Since $E\pi$ is 
contractible, $Y$ is the homotopy fiber of the map on the right, and hence 
homotopy equivalent to $X^{hH}$.  After taking homotopy fixed spaces 
$(-)^{h\pi}$, and since $(EG/H\times{}E\pi)/G\cong{}\widehat{B}G$, we get 
a new homotopy pullback square 
	\begin{diagram}[w=40pt]
	\llap{$(X^{hH})^{h\pi}\simeq$\,}Y^{h\pi} & \rTo & 
	\map(\widehat{B}G,X_{hG})_{\widehat{\widehat{B}\iota}} \\
	\dTo && \dTo \\
	\llap{$*\simeq$\,}E\pi^{h\pi} & \rTo & 
	\map(\widehat{B}G,\widehat{B}G)_{\widehat{B}\iota} 
	\rlap{~.} 
	\end{diagram}
Thus $(X^{hH})^{h\pi}$ is the homotopy fiber of the right hand map, hence 
equivalent to the space of sections of the fibration $X_{hG}\Right2{}BG$, 
which is $X^{hG}$.

Now assume that we are in the situation of (b).  In this situation, if 
$G$ acts on $X$, the restriction map $X^{hH}\Right2{\simeq}X^{hH_0}$ is 
$\pi$--equivariant, provided we use the models for $X^{hH}$ and $X^{hH_0}$ 
constructed above.  Taking homotopy fixed points on both sides for the 
action of $\pi$, we see that the inclusion $X^{hG}\Right2{}X^{hG_0}$ is a 
homotopy equivalence.

Alternatively, by Proposition \ref{p:X-hG-props}(a,b), it suffices to 
prove this for a genuine action of $G$ on $X$.  In this case,
	\[ X^{hG} = \map_G(EG,X) \simeq \map_G(EG\times E\pi, X) \cong 
	\map_{\pi}(E\pi,\map_H(EG,X))\,, 
	\]
where the last equivalence follows by adjunction.  We can identify 
$\map_H(EG,X)$ with $X^{hH}$, and thus $X^{hG}\simeq(X^{hH})^{h\pi}$.
In the situation of (b), we get a commutative square
	\begin{diagram}[w=60pt]
	\llap{$X^{hG}\simeq$\,} \map_G(EG\times{}E\pi,X) & \rTo & 
	\map_{\pi}(E\pi,\map_H(EG,X)) \rlap{\,$\simeq (X^{hH})^{h\pi}$} \\
	\dTo>{r_1} && \dTo>{r_2} & \hphantom{mmm} \\
	\llap{$X^{hG_0}\simeq$\,} \map_{G_0}(EG\times{}E\pi,X) & \rTo & 
	\map_{\pi}(E\pi,\map_{H_0}(EG,X)) \rlap{\,$\simeq 
	(X^{hH_0})^{h\pi}$\,.}
	\end{diagram}
where $r_1$ and $r_2$ are induced by restriction to $G_0$ or $H_0$.  Since 
$r_2$ is a homotopy equivalence by assumption (and by Proposition 
\ref{p:X-hG-props}(a)), $r_1$ is also a homotopy equivalence.  
\end{proof}

The following theorem deals with homotopy fixed points of actions of 
$K\times\Z$, where $K$ is a finite cyclic group of order prime to $p$.  
This can be applied when $K$ is a group of graph automorphisms of 
$BG\pcom$ (and $G$ is a compact connected Lie group), or when $K$ is a 
group of elements of finite order in $\Z_p^\times$ (for odd $p$).  

\begin{thm} \label{t:old3.8}
Fix a prime $p$.  Let $X$ be a connected, $p$-complete space such that
\begin{itemize}  
\item $H^*(X;\F_p)$ is noetherian, and 
\item $\Outt(X)$ is detected on $\widehat{H}^*(X;\Z_p)$.
\end{itemize}
Fix a finite cyclic group $K=\gen{g}$ of order $r$ prime to $p$, together 
with a proxy action $f\:X_{hK}\Right2{}BK$ of $K$ on $X$.  Let 
$\beta\:X_{hK}\Right2{}X_{hK}$ be a self homotopy equivalence such that 
$f\circ\beta=f$; and set $\alpha=\beta|_X$, a self homotopy equivalence of 
$X$.  Assume $H_*(\alpha;\F_p)$ is an automorphism of $H_*(X;\F_p)$ of 
$p$-power order (equivalently, the action is nilpotent).  Let 
$\kappa\:X\Right2{}X$ be the self homotopy equivalence induced by lifting 
a loop representing $g\in\pi_1(BK)$ to a homotopy of the inclusion 
$X\Right2{}X_{hK}$.  Then 
	\[ X^{h(\kappa\alpha)} \simeq (X^{hK})^{h\alpha}. \]
\end{thm}

\begin{proof}  Upon first replacing $X$ by the pullback of $X_{hK}$ and 
$EK$ over $BK$, and then by taking the double mapping telescope of the 
map from that space to itself induced by $\beta$, we can assume that $X$ 
has a genuine free action of the group $K\times\Z$, and that 
$\kappa$ and $\alpha$ are (commuting) homeomorphisms of $X$ which are the 
actions of generators of $K$ and of $\Z$.  In particular, 
$(\kappa\alpha)^r=\alpha^r$.

For each $k\ge1$, $\alpha$ acts on $H_*(X;\Z/p^k)$ as an 
automorphism of $p$-power order.  Since $r$ is prime to $p$, this implies 
that $\alpha$ and $\alpha^r$ generate the same closed subgroup of 
$\Outt(X)$.  So by Theorem \ref{t:Xha=Xhb}, the inclusion of 
$X^{h\alpha}$ into $X^{h\alpha^r}$ is a homotopy equivalence.  

By Proposition \ref{ZtoZpZr}(b), applied with 
$G=\gen{\kappa}\times\gen{\alpha}$, $G_0=\gen{\kappa\alpha}$, and 
$H=\gen{\alpha}$, the inclusion of $X^{h(K\times\gen{\alpha})}$ into 
$X^{h(\kappa\alpha)}$ is a homotopy equivalence.  By Proposition 
\ref{ZtoZpZr}(a), $X^{h(K\times\gen{\alpha})}\simeq(X^{hK})^{h\alpha}$, and 
this proves the theorem.  
\end{proof}

The following result comparing fusion systems of $G_2(q)$ and 
${}^3\!D_4(q)$, and those of $F_4(q)$ and ${}^2\!E_6(q)$,  is well known.  
For example, the first part follows easily from the lists of maximal 
subgroups of these groups in \cite{Kleidman1} and \cite{Kleidman2}, and 
also follows from the cohomology calculations in \cite{FM} and 
\cite{Milgram} (together with Theorem \ref{th:easyMP}).  We present it 
here as one example of how Theorem \ref{t:old3.8} can be applied.  


\begin{exmp} \label{G2-3D4}
Fix a prime $p$, and a prime power $q\equiv1$ (mod $p$).  Then the 
following hold.
\begin{enumerate}\renewcommand{\labelenumi}{\textup{(\alph{enumi})}}
\item If $p\ne3$, the fusion systems $\calf_p(G_2(q))$ and 
$\calf_p({}^3\!D_4(q))$ are isotypically equivalent.
\item If $p\ne2$, the fusion systems $\calf_p(F_4(q))$ and 
$\calf_p({}^2\!E_6(q))$ are isotypically equivalent.
\end{enumerate}
\end{exmp}


\begin{proof}  To prove (a), we apply Theorem \ref{t:old3.8}, with 
$X=B\Spin_8(\C)\pcom\simeq{}B\Spin(8)\pcom$, with $K\cong{}C_3$ having the 
action on $X$ induced by the triality automorphism, and with 
$\alpha=\Psi^q$ the unstable Adams operation on $X$.  

We first show that the inclusion of $G_2$ into $Spin(8)$ induces a 
homotopy equivalence $(BG_2)\pcom \simeq X^{hK}$.  Since there is always a 
map from the fixed point set of an action to its homotopy fixed point set, 
the inclusion of $G_2(\C)\cong{}Spin_8(\C)^K$ (cf.\ \cite[Theorem 
1.15.2]{GLS3}) into $Spin_8(\C)$ induces maps 
$(BG_2)\pcom\Right2{}X^{hK}\Right2{}X$.  The first map is a monomorphism 
in the sense of Dwyer and Wilkerson \cite[\S 3.2]{dw:fixpt}, since the 
composite is a monomorphism.  

By \cite[Theorem B(2)]{bm:chevalley}, $X^{hK}$ is the classifying space of 
a connected $2$-compact group.  Hence by \cite[Theorem 
B(2)]{bm:chevalley}, $H^*(X^{hK};\Q_p)$ is the polynomial algebra 
generated by the coinvariants $QH^*(X;\Q_p)_K$; i.e, the coinvariants of 
the $K$-action on the polynomial generators of $H^*(\Spin_8(\C);\Q_p)$.  
For any compact connected Lie group $G$ with maximal torus $T$, 
$H^*(BG;\Q)$ is the ring of invariants of the action of the Weyl group on 
$H^*(BT;\Q)$ \cite[Proposition 27.1]{Borel}, and is a polynomial algebra 
with degrees listed in \cite[Table VII]{sheptodd}.  In particular, 
$H^*(X;\Q_p)$ has polynomial generators are in degrees $4,8,12,8$,
and an explicit computation shows that $K$ fixes 
generators in degrees $4$ and $12$.  Thus 
$H^*(X^{hK};\Q_p)\cong{}H^*(BG_2(\C);\Q_p)$ (as graded $\Q_p$-algebras).  
It follows from \cite[Proposition 3.7]{mn:center} that 
$(BG_2)\pcom\Right2{}X^{hK}$ is an isomorphism of connected $2$-compact 
groups because it is a monomorphism and a rational isomorphism.

Now let $\kappa\in\Aut(X)$ generate the action of $K$.  By Theorem 
\ref{th:friedl}, 
	\[ X^{h(\kappa\alpha)}\simeq{}B({}^3\!D_4(q))\pcom 
	\qquad\textup{and}\qquad (X^{hK})^{h\alpha} \simeq 
	(BG_2\pcom)^{h\alpha}\simeq{}BG_2(q)\pcom. \]
Since $q\equiv1$ (mod $p$), the action of $\alpha=\Psi^q$ on $H^*(X;\F_p)$ 
has $p$-power order.  Hence $B({}^3\!D_4(q))\pcom\simeq{}BG_2(q)\pcom$ by 
Theorem \ref{t:old3.8}, and so these groups have isotypically equivalent 
$p$-fusion systems by Theorem \ref{th:easyMP}.  This proves (a).

Now set $X=BE_6(\C)\pcom$ and $K=\gen{\tau}$, where $\tau$ is an outer 
automorphism of order two.  For each $k\ge0$, $\tau$ acts on 
$H^{2k}(X;\Q_p)$ via $(-1)^k$:  this follows since $H^*(X;\Q_p)$ injects 
into the cohomology of any maximal torus and $\tau$ acts on an appropriate 
choice of maximal torus by via $(g\mapsto{}g^{-1})$.  Since $H^*(X;\Q_p)$ 
is polynomial with generators in degrees $4,10,12,16,18,24$, \cite[Theorem 
B(2)]{bm:chevalley} implies that $H^*(X^{hK};\Q_p)$ is polynomial with 
generators in degrees $4,12,16,24$, and hence is isomorphic to 
$H^*(BF_4(\C);\Q_p)$.  The rest of the proof of (b) is identical to that 
of (a).
\end{proof}


\newcommand{\pp}{\mathfrak{p}}

\section{Classical groups}

In the case of many of the classical groups, there is a much more 
elementary approach to Theorem \ref{ThA}.  Recall that the modular 
character $\chi_V$ of an $\F_q[G]$-module $V$ is defined by identifying 
$\F_q^{\times}$ with a subgroup of $\C^{\times}$, and then letting 
$\chi_V(g)\in\C$ (when $(|g|,q)=1$) be the sum of the eigenvalues of 
$V\xrightarrow{g}V$ lifted to $\C$.  We always consider this in the case 
where $G$ has order prime to $q$, and hence when two representations with 
the same character are isomorphic.  See \cite[\S18]{serre77} for more 
details.

For any finite group $G$, let $\Rep_n(G)$ be the set of isomorphism 
classes of $n$-dimensional irreducible complex representations (i.e., 
$\Rep_n(G)=\Rep(G,GL_n(\C))$ in the notation used elsewhere).  For any 
prime $p$ and any $q$ prime to $p$, 
$\widebar{\gen{q}}\subseteq(\widehat{\Z}_p)^\times$ denotes the closure of 
the subgroup generated by $q$.

In the following theorem, we set $GL_n^+(q)=GL_n(q)$ and 
$GL_n^-(q)=GU_n(q)$ for convenience.  

\begin{prop} \label{t:Fp(class)}
Fix a prime $p$, and let $q$ be a prime power which is prime to $p$.  
\begin{enumerate}\renewcommand{\labelenumi}{\textup{(\alph{enumi})}}

\item Fix $n\ge2$ and $\epsilon=\pm1$.  For any finite $p$-group $P$, 
$\Rep(P,GL_n^\epsilon(q))$ can be identified with the set of those 
$V\in\Rep_n(P)$ such that $\chi_V(g^{\epsilon{}q})=\chi_V(g)$ for all 
$g\in{}P$.  


\item Assume $p$ is odd and fix $n\ge1$.  $G=Sp_{2n}(q)$ and 
$G_1=GO_{2n+1}(q)$.  Then for any finite $p$-group $P$, 
$\Rep(P,Sp_{2n}(q))$ and $\Rep(P,GO_{2n+1}(q))$ can be identified with the 
set of those $V\in\Rep_{2n}(P)\cong\Rep_{2n+1}(P)$ such that 
$\chi_V(g^q)=\chi_V(g)=\chi_V(g^{-1})$ for all $g\in{}P$.  In particular, 
the fusion systems $\calf_p(Sp_{2n}(q))$ and $\calf_p(GO_{2n+1}(q))$ are 
isotypically equivalent.  

\end{enumerate}
\end{prop}

\begin{proof} Let $K\subseteq\C$ be the subfield generated by all $p$-th 
power roots of unity.  For each $r\in\Z_p^\times$, let $\psi^r\in\Aut(K)$ 
be the field automorphism $\psi^r(\zeta)=\zeta^r$ for each root of unity 
$\zeta$.  

\noindent\textbf{(a) }  Let $\widehat{K}$ be the extension of $\Q_q$ by 
all roots of unity prime to $q$, let $A\subseteq\Q_q$ be the ring of 
integers, and let $\pp\subseteq{}A$ be the maximal ideal.  Thus 
$A/\pp\cong\fqbar$.  By modular representation theory (cf.\ \cite[Theorems 
33 \& 42]{serre77}), for each finite $p$-group $P$, there is an 
isomorphism of representation rings 
$R_{\widehat{K}}(P)\Right3{\cong}R_{\fqbar}(P)$, which sends the class of 
a $\widehat{K}[P]$-module $V$ to $M/\pp{}M$ for any $P$-invariant 
$A$-lattice $M\subseteq{}V$.  This clearly sends an actual representation 
to an actual representation.  If $M_1\subseteq{}V_1$ and 
$M_2\subseteq{}V_2$ are such that $M_1/\pp{}M_1$ and $M_2/\pp{}M_2$ have 
an irreducible factor in common, then since $|P|$ is invertible in 
$\fqbar$ and in $A$, any nonzero homomorphism 
$\varphi\in\Hom_{P}(M_1/\pp{}M_1,M_2/\pp{}M_2)$ can be lifted (by 
averaging over the elements of $P$) to a homomorphism 
$\widehat{\varphi}\in\Hom_P(V_1,V_2)$.  From this we see that the 
isomorphism $R_{\widehat{K}}(P)\cong{}R_{\fqbar}(P)$ restricts to a 
bijection between irreducible representations, and also between 
$n$-dimensional representations for any given $n$.  So 
	\[ \Rep(P,GL_n(\fqbar))\cong \Rep(P,GL_n(\widehat{K})) \cong
	\Rep(P,GL_n(K)) \cong \Rep(P,GL_n(\C)), \]
where the last two bijections follow from \cite[Theorem 24]{serre77}.  

The centralizer of any finite $p$-subgroup of $GL_n(\fqbar)$ is a product 
of general linear groups, and hence connected.  Thus by Proposition 
\ref{gls3-215}, $\Rep(P,GL_n^\epsilon(q))$ injects into 
$\Rep(P,GL_n(\fqbar))$, and its image is the set of representations which 
are fixed by the Steinberg endomorphism $\psi^{\epsilon{}q}$ on 
$GL_n(\fqbar)$.  

Fix $V\in\Rep(P,GL_n(\fqbar))$ and $g\in{}P$, let 
$\xi_1,\dots,\xi_n\in\fqbar$ be the eigenvalues of the action of $g$ on 
$V$, and let $\zeta_1,\dots,\zeta_n\in{}K$ be the corresponding $p$-th power 
roots of unity in $\C$.  Then $\psi^q(V)$ has eigenvalues 
$\xi_1^q,\dots,\xi_n^q$, $\psi^{-1}(V)=V^*$ has eigenvalues 
$\xi_1^{-1},\dots,\xi_n^{-1}$, and so $\psi^{\epsilon{}q}(V)$ has eigenvalues 
$\xi_1^{\epsilon{}q},\dots,\xi_n^{\epsilon{}q}$.  This proves that 
	\[ \chi_{\psi^{\epsilon{}q}(V)}(g)= 
	\chi_1^{\epsilon{}q}+\ldots+\chi_n^{\epsilon{}q} 
	=\chi_V(g^{\epsilon{}q}) \]
for all $V\in\Rep_n(P)$ and all $g\in{}P$.  Thus $V\in\Rep_n(P)$ is in the 
image of $\Rep(P,GL_n^\epsilon(q))$ if and only if 
$\chi_V(g^{\epsilon{}q})=\chi_V(g)$ for all $g\in{}P$.  

\smallskip

\noindent\textbf{(b) }  Now assume $p$ is odd, and let $P$ be a finite 
$p$-group.  For any irreducible $\fqbar[P]$-representation $W$ which is 
self dual, $\sum_{g\in{}P}\chi_W(g^2)\ne0$:  this is shown in 
\cite[Proposition II.6.8]{BtD} for complex representations, and the same 
proof applies in our situation.  
Since $|P|$ is odd, this implies that $\sum_{g\in{}P}\chi_W(g)\ne0$, and 
hence that $W$ is the trivial representation.  In other words, the only 
self dual irreducible representation is the trivial one.  Hence every self 
dual $\fqbar[P]$-representation $V$ has the form 
$V=V_0\oplus{}W\oplus{}W'$, where $P$ acts trivially on $V_0$ and has no 
fixed component on $W$ and $W'$, $W'\cong{}W^*$, and 
$\Hom_{\fqbar[P]}(W,W')=0$. 

Fix a self dual $\fqbar[P]$-representation $V$, and write 
$V=V_0\oplus{}W\oplus{}W'$ as above.  Fix $\epsilon=\pm1$, where 
$\epsilon=+1$ if $\dim_{\fqbar}(V)$ is odd, and write 
``$\epsilon$-symmetric'' to mean symmetric ($\epsilon=+1$) or 
symplectic ($\epsilon=-1$).  For any nondegenerate 
$\epsilon$-symmetric form $\bb_0$ on $V_0$ and any 
$\fqbar[P]$-linear isomorphism $f\:W'\Right2{\cong}W^*$, there is a 
nondegenerate $\epsilon$-symmetric form $\bb$ on $V$ defined by
	\beqq \bb((v_1,w_1,w'_1),(v_2,w_2,w'_2)) = \bb_0(v_1,v_2) + 
	f(w'_1)(w_2) +\epsilon f(w'_2)(w_1) \label{e:bb} \eeqq
for $v_i\in{}V_0$, $w_i\in{}W$, and $w'_i\in{}W'$.  Conversely, if $\bb$ 
is any nonsingular $\epsilon$-symmetric form on $V$, then $\bb$ must be 
nonsingular on $V$, and zero on $W$ and on $W'$, and hence has the form 
\eqref{e:bb} for some $\bb_0$ and $f$.  Since all such forms are 
isomorphic, this proves that $\Rep(P,Sp_{2n}(\fqbar))$ can be identified 
with the set of self dual elements of $\Rep_{2n}(P)$, and 
$\Rep(P,GO_{2n+1}(\fqbar))$ with the set of self dual elements of 
$\Rep_{2n+1}(P)$.  

Thus $\Rep(P,Sp_{2n}(\fqbar))\subseteq\Rep_{2n}(P)$ and 
$\Rep(P,GO_{2n+1}(q))\subseteq\Rep_{2n+1}(P)$ are both the sets of self 
dual elements.  Since 
$GO_{2n+1}(\fqbar)=SO_{2n+1}(\fqbar)\times\gen{\pm\Id}$, 
$\Rep(P,SO_{2n+1}(\fqbar))=\Rep(P,GO_{2n+1}(\fqbar))$.  Also, as we just 
saw, each odd dimensional self dual $P$-representation has odd dimensional 
fixed component, and thus 
$\Rep(P,SO_{2n+1}(\fqbar))=\Rep(P,Sp_{2n}(\fqbar))$ as subsets of 
$\Rep_{2n+1}(P)$.  

We claim that the centralizer of any finite $p$-subgroup of 
$Sp_{2n}(\fqbar)$ or $SO_{2n+1}(\fqbar)$ is connected.  To see this, fix 
such a subgroup $P$, let $V$ be the corresponding representation with 
symmetric or symplectic form $\bb$, and let $V=V_0\oplus{}W\oplus{}W'$ be 
a decomposition such that $\bb$ is as in \eqref{e:bb}.  Then the 
centralizer of $P$ in $\Aut(V,\bb)$ is the product of $\Aut(V_0,\bb_0)$ 
with $\Aut(W)$, and hence its centralizer in $Sp_{2n}(\fqbar)$ or 
$SO_{2n+1}(\fqbar)$ is connected.

We now apply Proposition \ref{gls3-215}, exactly as in the proof of 
(a), to show that for a $p$-group $P$, $\Rep(P,Sp_{2n}(q))$ injects into 
$\Rep_{2n}(P)$ with image the set of those $V$ with 
$\chi_V(g)=\chi_V(g^q)=\chi_V(g^{-1})$ for all $g\in{}P$; and similarly 
for $\Rep(P,SO_{2n+1}(q))$.  
\end{proof}

For the linear and unitary groups, Theorem \ref{ThA} follows immediately 
from Proposition \ref{t:Fp(class)}(a).  Also, Proposition 
\ref{t:Fp(class)}(b) implies that when $p$ is odd, Theorem \ref{ThA} holds 
for the symplectic and odd orthogonal groups; and also that 
	\[ \calf_p(Sp_{2n}(q)) \simeq \calf_p(GO_{2n+1}(q)) 
	\bigl( \simeq \calf_p(SO_{2n+1}(q)) \bigr) \]
for each odd $p$, each $n\ge1$, and each $q$ prime to $p$.  Theorem 
\ref{ThA} for the even orthogonal groups then follows from the following 
observation.

\begin{prop} \label{even-odd}
For each odd prime $p$, each prime power $q$ prime to $p$, and each 
$n\ge1$, 
	\begin{align*}  
	\calf_p(GO_{2n}^\pm(q)) &\simeq \calf_p(SO_{2n+1}(q))
	\simeq \calf_p(Sp_{2n}(q)) & 
	&\textup{if $q^n\not\equiv\mp1$ (mod $p$)} \\
	\calf_p(GO_{2n}^\pm(q)) &\simeq \calf_p(SO_{2n-1}(q))
	\simeq \calf_p(Sp_{2n-2}(q)) & 
	&\textup{if $q^n\not\equiv\pm1$ (mod $p$)~.} 
	\end{align*}
\end{prop}

\begin{proof}  Any inclusion $GO_k^\pm(q)\le{}GO_{k+1}^\pm(q)$ induces an 
injection of $\Rep(P,GO_k^\pm(q))$ into $\Rep(P,GO_{k+1}^\pm(q))$ for each 
$p$-group $P$.  Thus $\calf_p(GO_k^\pm(q))$ is a full subcategory of 
$\calf_p(GO_{k+1}^\pm(q))$ by Proposition \ref{full-subcat}.  Hence we get 
an equivalence of $p$-fusion categories whenever $GO_k^\pm(q)$ has index 
prime to $p$ in $GO_{k+1}^\pm(q)$.  By the standard formulas for the 
orders of these groups, 
	\begin{align*}  
	[GO_{2n+1}(q):GO^\pm_{2n}(q)] &= q^n(q^n\pm1) \\
	[GO_{2n}^\pm(q):GO_{2n-1}(q)] &= q^{n-1}(q^n\mp1) ~,
	\end{align*}
and the proposition follows.
\end{proof}

Another consequence of Proposition \ref{t:Fp(class)} is the following:

\begin{prop} \label{red2gl}
Fix an odd prime $p$, and a prime power $q$ prime to $p$.  Set 
$s=\order(q)$ mod $p$.  
\begin{enumerate}\renewcommand{\labelenumi}{\textup{(\alph{enumi})}}
\item If $s$ is even, then for each $n\ge1$, the inclusion 
$Sp_{2n}(q)\le{}GL_{2n}(q)$ induces an equivalence 
$\calf_p(Sp_{2n}(q))\simeq\calf_p(GL_{2n}(q))$ of fusion systems.  
\item If $s\equiv2$ (mod $4$), then for each $n\ge1$, 
$Sp_{2n}(q)\le{}Sp_{2n}(q^2)$ induces an equivalence of $p$-fusion systems.
\end{enumerate}
\end{prop}

\begin{proof}  If $s$ is even, then $-1$ is a power of $q$ modulo $p$, and 
also modulo $p^n$ for all $n\ge2$.  Hence if $P$ is a $p$-group, and 
$V\in\Rep_{2n}(P)$ is such that $\chi_V(g)=\chi_V(g^q)$ for all $g\in{}P$, 
then also $\chi_V(g)=\chi_V(g^{-1})$ for all $g$.  So by Proposition 
\ref{t:Fp(class)}, $\Rep(P,Sp_{2n}(q))\cong\Rep(P,GL_{2n}(q))$, and so (a) 
follows from Proposition \ref{p:isot<=>fpres}(a,b). 

If $s\equiv2$ (mod $4$), then $q^2$ has odd order in $(\Z/p)^\times$, and 
hence in $(\Z/p^n)^\times$ for all $n$.  So $\gen{q}=\gen{q^2,-1}$ in 
$(\Z/p^n)^\times$ for all $n$.  Thus for a $p$-group $P$ and 
$V\in\Rep_k(P)$, $\chi_V(g)=\chi_V(g^q)=\chi_V(g^{-1})$ for all $g\in{}P$ 
if and only if $\chi_V(g)=\chi_V(g^{q^2})=\chi_V(g^{-1})$ for all $g$.  
So (c) follows from Proposition \ref{p:isot<=>fpres}(b). 
\end{proof}

Upon combining Propositions \ref{even-odd} and \ref{red2gl} with Theorem 
\ref{ThA}, we see that for $p$ odd and $q$ prime to $p$, each of the 
fusion systems $\calf_p(Sp_{2n}(q))\simeq\calf_p(SO_{2n+1}(q))$ and 
$\calf_p(GO_{2n}^\pm(q))$ is isotypically equivalent to the $p$-fusion 
system of some general linear group.  Note, for example, that when $q$ has 
odd order in $(\Z/p)^\times$, there is some $q'$ such that 
$\widebar{\gen{q}}=\widebar{\gen{q'{}^2}}$ in $\Z_p^\times$, and so 
	\[ \calf_p(Sp_{2n}(q)) \simeq \calf_p(Sp_{2n}(q'{}^2)) \simeq 
	\calf_p(Sp_{2n}(q')) \simeq \calf_p(GL_{2n}(q')). \]
Since $\calf_p(SO_{2n}^\pm(q))$ is always normal of index at most two in 
$\calf_p(SO_{2n}^\pm(q))$, this also gives a description of those fusion 
systems in terms of fusion systems of general linear groups.






\end{document}